\documentclass[11pt]{amsart}

\usepackage{color,enumitem}
\usepackage{amsmath,amssymb,bbm,bm}
\usepackage{amscd,tikz-cd}
\usepackage{mathrsfs}
\usepackage{amsthm}
\usepackage{todonotes}

\theoremstyle{plain}

\newtheorem{theorem}{\bf Theorem}[section]
\newtheorem{lemma}[theorem]{\bf Lemma}
\newtheorem{proposition}[theorem]{\bf Proposition}
\newtheorem{corollary}[theorem]{\bf Corollary}

\newtheorem{thmx}{Theorem}

\theoremstyle{definition}
\newtheorem{definition}[theorem]{\bf Definition}

\newtheorem{remark}[theorem]{\bf Remark}

\newcommand{\eqa}[1]{
\begin{align*}
#1
\end{align*}}

\newcommand{\Nat}{\mathbb{N}}
\newcommand{\Rea}{\mathbb{R}}

\newcommand{\Com}{\mathbb{C}}

\newcommand{\ri}{{\rm i}}

\newcommand{\T}{\mathbb{T}}

\newcommand{\Ze}{\mathcal{Z}}

\newcommand{\Q}{\mathbb{Q}}
\newcommand{\R}{\mathbb{R}}
\newcommand{\C}{\mathbb{C}}
\newcommand{\N}{\mathbb{N}}

\newcommand{\Aut}{\mathrm{Aut}}

\usepackage[all]{xy}

\title{Lie theoretic approach to unitary groups of $C^*$-algebras}
\author[H. Ando]{Hiroshi Ando}
\address{Hiroshi Ando, Department of Mathematics and Informatics, Chiba University, 1-33 Yayoi-cho, Inage, Chiba, 263- 8522,
Japan}
\email{hiroando@math.s.chiba-u.ac.jp}

\author[M. Doucha]{Michal Doucha}
\address{Institute of Mathematics\\
Czech Academy of Sciences\\
\v Zitn\'a 25\\
115 67 Praha 1\\
Czech Republic}
\email{doucha@math.cas.cz}
\urladdr{https://users.math.cas.cz/~doucha/}

\thanks{H. Ando was supported by Japan Society for the Promotion of Sciences (20K03647). M. Doucha was supported by the GA\v{C}R project 22-07833K and by the Czech Academy of Sciences (RVO 67985840).}

\begin{document}
\maketitle
\begin{abstract}
Following Robert's \cite{Rob19}, we study the structure of unitary groups and groups of approximately inner automorphisms of unital $C^*$-algebras, taking advantage of the former being Banach-Lie groups. For a given unital $C^*$-algebra $A$, we provide a description of the closed normal subgroup structure of the connected component of the identity of the unitary group, denoted by $U_A$, resp. of the subgroup of approximately inner automorphisms induced by the connected component of the identity of the unitary group, denoted by $V_A$, in terms of perfect ideals, i.e. ideals admitting no characters. When the unital algebra is locally AF, we show that there is a one-to-one correspondence between closed normal subgroups of $V_A$ and perfect ideals of the algebra, which can be in the separable case conveniently described using Bratteli diagrams; in particular showing that every closed normal subgroup of $V_A$ is perfect. We also characterize unital $C^*$-algebras $A$ such that $U_A$, resp. $V_A$ are topologically simple, generalizing the main results from \cite{Rob19}. In the other way round, under certain conditions, we characterize simplicity of the algebra in terms of the structure of the unitary group. This in particular applies to reduced group $C^*$-algebras of discrete groups and we show that when $A$ is a reduced group $C^*$-algebra of a non-amenable countable discrete group, then $A$ is simple if and only if $U_A/\T$ is topologically simple.
\end{abstract}

\section{Introduction}
Unitary groups of unital $C^*$-algebras and von Neumann algebras are objects that have been studied from several different perspectives in various subjects such as operator algebras, infinite dimensional differential geometry, mathematical physics, and topological group theory. Even within the $C^*$-algebra theory, the interest is both in the intrinsic structure of these groups, such as the normal subgroup structure when the corresponding algebra is simple, or in their relation to areas such as K-theory (see e.g., the monograph \cite{RLLbook}) and exponential rank and length (see \cite{Ph93} and references therein). It was rarely used though that when equipped with the norm topology these groups are linear Banach-Lie groups, i.e. very well-behaved, in general, infinite-dimensional Lie groups for which the powerful machinery of Lie theoretic techniques is available. A rare innovative exception, from which also our inspiration comes, was the paper \cite{Rob19} of Robert which shows that the closed commutator subgroup modulo its center of the connected component of the unit of a unitary group of any simple unital $C^*$-algebra is topologically simple using Lie theoretic methods. Robert also showed that the group of approximately inner automorphisms induced by unitaries from the connected component of the unit of any simple unital $C^*$-algebra is topologically simple resolving a long standing open problem from \cite{ER93}. For another recent work related to unitary groups of $C^*$-algebras from the Lie perspective, we refer e.g., to \cite{ChaRob23,TakRobarxiv23,ADM22}.

We remark that Robert himself was inspired by a result \cite{BKS08} on Lie ideals in $C^*$-algebras, another area where Lie theory meets $C^*$-algebras and which has been already extensively studied (see e.g., \cite{MarMur98,Mar10,HopPau04,Robert16,GaTh23,GupJain}).

It is the goal of this paper to continue in this line of research and study the connected component of the unit of unitary groups, and their closely related cousins: the groups of approximately inner automorphisms induced by unitaries from such components, using Lie group techniques.

We refine Robert's description of normal subgroup structure of these groups from \cite{Rob19} and relate them to \emph{perfect ideals}, i.e. closed two-sided ideals admitting no characters. For a unital $C^*$-algebra $A$, denote by $U_A$ the connected component of the unit of the unitary group of $A$ (equipped with the norm topology) and by $V_A$ the closure, in $\Aut(A)$, of the inner automorphisms coming from $U_A$. For a closed two-sided ideal $I\subseteq A$, denote by $U_I$ the connected component of the identity of the group $\{u\in U(A)\mid u-1\in I\}$ and by $CU_I$ the closed commutator subgroup of $U_I$, which are both closed subgroups of $U_A$. Similarly, denote by $V_I$ the subgroup $\overline{\{\mathrm{Ad}(u)\mid u\in U_I\}}\leq V_A$.

\begin{thmx}[see Theorems~\ref{thm:perfectsubgrpsofU_A}, \ref{thm:setofnormalsubgrps}, \ref{thm:normalsubgrpsofV_A} and Corollary~\ref{cor:perfectsubgrpsofV_A}]
Let $A$ be a unital $C^*$-algebra. The maps \[I\to U_I,\quad I\to V_I\] are one-to-one correspondences between perfect ideals of $A$ and perfect closed normal subgroups of $U_A$, respectively of $V_A$. In fact, these maps are `almost' one-to-one correspondences between perfect ideals and general closed normal subgroups of $U_A$ and $V_A$.
\end{thmx}
The vague wording `almost one-to-one' will be made precise in Theorems~\ref{thm:setofnormalsubgrps} and~\ref{thm:normalsubgrpsofV_A}.

For some $C^*$-algebras $A$ we obtain precise description of the closed normal subgroups of $V_A$.
\begin{thmx}[see Theorems~\ref{thm:AFalgebras} and~\ref{thm:Bratteli}]
Let $A$ be a unital locally AF algebra. Then the map \[I\to V_I\] is a one-to-one correspondence between perfect ideals of $A$ and closed normal subgroups of $V_A$. Moreover, when $A$ is separable, as a consequence, there is a one-to-one correspondence between closed normal subgroups of $V_A$ and directed and hereditary subsets of a Bratteli diagram of $A$ that do not contain an infinite path of $1$s.
\end{thmx}
We refer the reader to Theorem~\ref{thm:Bratteli} for a precise description of the correspondence between closed normal subgroups of $V_A$ and such subsets of a Bratteli diagram of $A$.

Denoting by $V(A)$ the group of all approximately inner automorphims of a unital $C^*$-algebra $A$, we have the normal series $V_A\trianglelefteq V(A)\trianglelefteq \Aut(A)$ and the structure of the quotients $\Aut(A)/V(A)$, resp. $V(A)/V_A$ can be understood in some cases (see e.g., \cite{ER93}), we emphasize that it is indeed important to understand the structure of $V_A$ when one is interested in $\Aut(A)$.

The rest of the paper concerns the following problems. First we note that the converse to Robert's theorem, i.e. that when the algebra $A$ is simple, then the closed commutator subgroup of $U_A$ modulo the center and $V_A$ are topologically simple, does not hold. We provide a characterization of those $C^*$-algebras for which these groups are topologically simple. Then we turn to the issue whether the simplicity of the algebra can be detected from the corresponding unitary group structure. We provide several partial results in this direction that in particular apply to the reduced group $C^*$-algebras of discrete groups, and among other things we prove the following.
\begin{thmx}[see Theorem~\ref{thm:simplicity of C_r(G)}]
Let $G$ be a countable discrete non-amenable group. Then $A:=C^*_r(G)$ is simple if and only if $U_A/\T$ is topologically simple.
\end{thmx}

The paper is organized as follows. The following Section~\ref{sect2} introduces all the necessary notions and notations. Section~\ref{sect3} contains first preliminary results that are used in the sequel and Sections~\ref{sect4} and~\ref{sect5} concern the normal subgroup structure of $U_A$ and $V_A$ respectively. In particular, Theorems A and B are proved there. Section~\ref{sect6} contains some additional information about the Lie structure of certain normal subgroups of $U_A$ which in particular shows that subalgebras that are $V_A$-invariant are also $V(A)$-invariant. Section~\ref{sect7} focuses on the simplicity of $C^*$-algebras and Theorem C is proved there.
\section{Preliminaries and notation}\label{sect2}
In this short section, we recall basic facts from (infinite-dimensional) Lie theory and $C^*$-algebras. We also introduce here all the main notation that we shall work with throughout the whole paper.
\subsection{Infinite-dimensional Lie groups}
A \emph{Banach-Lie group} is a group that is also a Banach manifold, i.e. an infinite dimensional smooth manifold modelled on a Banach space, and such that the group operations are smooth. In this paper, we solely work with a much more concrete and tractable subclass of Banach-Lie groups that we define below.

\begin{definition}\label{def:Liegrp}
Let $A$ be a unital Banach algebra and denote by $GL(A)$ the group of invertible elements of $A$ equipped with the norm topology. A closed subgroup $G\subseteq GL(A)$ is a \emph{linear Lie subgroup} if the map $\exp:A\to GL(A)$, $a\to \sum_{k=0}^\infty \frac{a^k}{k!}$, is a local homeomorphism around $0$ between $\mathfrak{L}(G):=\{a\in A\mid \exp(ta)\in G\,\forall t\in\Rea\}$ and $G$.
\end{definition}

It follows from \cite[Definition 5.32 and Corollary 5.34]{HofMor} that $G$ is then a Banach-Lie group modelled on $\mathfrak{L}(G)$, which is a Banach-Lie algebra of $G$. That is, a Banach subspace of $A$ closed under the bracket operation \[[x,y]=xy-yx,\; x,y\in \mathfrak{L}(G).\]

In particular, the following properties hold true for $G$ and $\mathfrak{L}(G)$ that we shall use without further reference in the sequel.

\begin{itemize}
    \item The connected component $G_0$ of the unit $1_G$ of $G$ is also a linear Lie group and $\mathfrak{L}(G_0)=\mathfrak{L}(G)$.
    \item For each $a\in \mathfrak{L}(G)$, $\exp(a)\in G_0$. Conversely, for each $g\in G_0$ there are $a_1,\ldots,a_n\in\mathfrak{L}(G)$ such that $g=\prod_{i=1}^n \exp(a_i)$.
    \item (The Trotter product formula) \[\exp(a+b)=\lim_{n\to\infty} \Big(\exp(\tfrac{1}{n}a)\exp(\tfrac{1}{n}b)\Big)^n,\] for every $a,b\in\mathfrak{L}(G)$.
    \item (The commutator formula) \[\exp([a,b])=\lim_{n\to\infty} \Big(\exp(\tfrac{1}{n}a)\exp(\tfrac{1}{n}b)\exp(-\tfrac{1}{n}a)\exp(-\tfrac{1}{n}b)\Big)^{n^2},\] for every $a,b\in\mathfrak{L}(G)$.
    \item (The Baker-Campbell-Hausdorff formula) \[\exp(a)\exp(b)=\exp\big(a+b+\frac{1}{2}[a,b]+\frac{1}{12}[a,[a,b]]-\frac{1}{12}[b,[a,b]]+\ldots\big),\] for all sufficiently small $a,b\in\mathfrak{L}(G)$. In fact, by \cite[Theorem 5.56]{BonfiglioliBCHMR2883818}, the right hand side converges whenever $a,b\in \mathfrak{L}(G)$ satisfy $\|a\|+\|b\|<\tfrac{1}{2}\log 2$.
\end{itemize}

Notice that in particular, $GL(A)$ is a linear Banach-Lie group and $\mathfrak{L}(GL(A))=A$.

\subsection{Unitary groups of $C^*$-algebras}
Let $A$ be a unital $C^*$-algebra. 
By $\Ze(A)$ we shall denote the center of $A$. Similarly, if $G$ is a (topological) group, by $\Ze(G)$ we shall denote the center of $G$. For two subsets $I,J\subseteq A$ we denote by $[I,J]$ the linear span of the commutators $\{xy-yx\mid x\in I,y\in J\}$ in $A$. Analogously, for subgroups $H_1,H_2$ of $G$ the subgroup of $G$ generated by commutators $(h_1,h_2)=h_1h_2h_1^{-1}h_2^{-1},\,\,h_1\in H_1,\,h_2\in H_2$ is denoted by $(H_1,H_2)$. 
Denote by $U(A)$ the unitary group, i.e. $U(A)=\{u\in A\mid uu^*=u^*u=1\}\subseteq {\rm GL}(A)$. Then $U(A)$ is a linear Lie group and its Lie algebra $\mathfrak{L}(U(A))$ is by definition the set of skew-adjoint elements, i.e. the set $\{a\in A\mid a^*=-a\}$. However, here we follow the standard practice from $C^*$-algebra theory and mathematical physics and we identify $\mathfrak{L}(U(A))$ with the set $A_{{\rm sa}}=\{a\in A\mid a^*=a\}$ of self-adjoint elements. Note that as a result we work with the exponential map $a\in A_{{\rm sa}}\to \exp(\ri a)\in U(A)$ and the bracket operation $\ri [x,y]$, for $x,y\in A_{{\rm sa}}$. We also emphasize that we shall interchangeably write both $\exp(a)$ and $e^a$ when using the exponential map as the latter is sometimes notationally more convenient when writing down formulas.

For $S\subseteq A$, we denote by ${\rm Id}_A(S)$ the closed two-sided ideal of $A$ generated by $S\subseteq A$. 

The connected component of the unit in $U(A)$ will be denoted by $U_A$. It is therefore again a linear Lie subgroup with $A_{{\rm sa}}$ as a Lie algebra. Moreover, if $I\subseteq A$ is a closed two-sided ideal, we define $U(I):=\{u\in U(A)\mid u-1\in I\}$ and then $U_I$ to be the connected component of the unit of $U(I)$. We have that $U_I$ is a closed connected normal linear Lie subgroup of $U_A$ and $\mathfrak{L}(U_I)=I_{{\rm sa}}:=A_{{\rm sa}}\cap I$. This follows immediately from Definition~\ref{def:Liegrp} and the fact that the exponential map $a\to e^{\ri a}$ maps $I$ into $1+I$ and the logarithm map, whenever defined, maps $1+ I$ into $I$. The closure of the commutator subgroup $(U_I,U_A)$ (resp. $(U_I,U_I)$) is denoted by $CU_{I,A}$ (resp. $CU_I$). These two groups are actually equal and topologically perfect. The latter is proved, for $CU_I$, in \cite[Theorem 3.6]{Rob19}, the former is implicit in \cite{Rob19} and explicitly proved in Lemma~\ref{lem:U_IU_I=U_AU_A}.
\bigskip

Denote by $\Aut(A)$ the group of all automorphisms of $A$ equipped with the pointwise convergence topology. For each $u\in U(A)$, $\mathrm{Ad}(u):A\to A$, defined by $a\to uau^*$, is an automorphism called an \emph{inner automorphism}. Clearly, $\{\mathrm{Ad}(u)\mid u\in U(A)\}\subseteq \Aut(A)$ is a subgroup, which is not closed in general. We denote by $V(A)$ the closure $\overline{\{\mathrm{Ad}(u)\mid u\in U(A)\}}$ in $\Aut(A)$, which is called the group of \emph{approximately inner automorphisms}. We shall also denote by $V_A$ the closure $\overline{\{\mathrm{Ad}(u)\mid u\in U_A\}}$. Notice that we have a normal series $V_A\trianglelefteq V(A)\trianglelefteq \Aut(A)$.

For any closed two-sided ideal $I\subseteq A$, the groups $V(I)$ and $V_I$ are defined analogously.
\section{First results}\label{sect3}
In this section, we collect and prove several basic results that will become handy in the sequel.
\subsection{Ideals and commutator subgroups}
The next lemma is proved in \cite[Proposition 5.25]{BKS08}, by an argument using W$^*$-algebras and a result from \cite{Meiers81MR0638381}. As is mentioned in the last paragraph to \cite[Lemma 1.4]{Robert16}, it can be shown using quasi-central approximate units. Here we include such a proof.   
\begin{lemma}\label{lem: liecom}
Let $A$ be a unital $C^*$-algebra and $I$ be a closed two-sided ideal of $A$. Then $\overline{[A,A]}\cap I=\overline{[I,I]}$ holds. 
\end{lemma}
\begin{proof}
To ease notation for $x,y\in A$ and $\varepsilon>0$, we denote $x\stackrel{\varepsilon}{\sim}y$ to mean $\|x-y\|<\varepsilon$. 
It is clear that $\overline{[I,I]}\subseteq \overline{[A,A]}\cap I$. 
Let $x\in \overline{[A,A]}\cap I$ and $\varepsilon>0$. 
Then there exists $x_0\in [A,A]$ such that $x\stackrel{\varepsilon}{\sim}x_0$. Let $(e_{\lambda})_{\lambda\in \Lambda}$ be an approximate units for $I$ that is quasicentral in $A$, meaning that $\lim_{\lambda}\|ae_{\lambda}-e_{\lambda}a\|=0$ for every $a\in A$ (see e.g., \cite[Theorem I.9.16]{Davidson-book}). 
We will use for functions $f,g\colon \Lambda \to A$ the notation $f(\lambda)\stackrel{\lambda \to \infty}{\sim}g(\lambda)$ to mean $\lim_{\lambda}\|f(\lambda)-g(\lambda)\|=0$. 
By $x\in I$, there exists $\lambda_0\in \Lambda$ such that $e_{\lambda}xe_{\lambda}\stackrel{\varepsilon}{\sim}x$ for every $\lambda\ge \lambda_0$. Because $e_{\lambda}$'s are positive contractions, we also have $e_{\lambda}xe_{\lambda}\stackrel{\varepsilon}{\sim}e_{\lambda}x_0e_{\lambda}\,(\lambda \ge \lambda_0)$. Write 
\[x_0=\sum_{k=1}^n(a_kb_k-b_ka_k),\,\,a_k,b_k\in A\,(1\le k\le n).\]
Because $(e_{\lambda})_{\lambda}$ is quasicentral in $A$, we have 
\[e_{\lambda}x_0e_{\lambda}\stackrel{\lambda\to \infty}{\sim}\sum_{k=1}^n(e_{\lambda}a_ke_{\lambda}b_k-e_{\lambda}b_ke_{\lambda}a_k),\]
so that there exists $\lambda_1\ge \lambda_0$ such that 
\[e_{\lambda}x_0e_{\lambda}\stackrel{\varepsilon}{\sim}\sum_{k=1}^n[e_{\lambda}a_k,e_{\lambda}b_k]=:x_1(\lambda)\in [I,I],\,\,\,(\lambda \ge \lambda_1).\]
Then for $\lambda\ge \lambda_1$, 
\[x\stackrel{\varepsilon}{\sim}e_{\lambda}xe_{\lambda}\stackrel{\varepsilon}{\sim}e_{\lambda}x_0e_{\lambda}\stackrel{\varepsilon}{\sim}x_1(\lambda)\in [I,I].\]
Since $\varepsilon$ is arbitrary, we obtain $x\in \overline{[I,I]}$. 
\end{proof}
\begin{lemma}\label{lem:U_IU_I=U_AU_A}
Let $A$ be a unital $C^*$-algebra and $I$ be a closed two-sided ideal of $A$. Then $CU_I=CU_{I,A}$.
\end{lemma}

\begin{proof}
The inclusion $CU_I\subseteq CU_{I,A}$ is obvious. To show the converse inclusion, note that by \cite[Lemma 3.1]{Rob19}, the set $\{e^{\ri c}\mid c\in [I,A]\cap A_{\rm sa}\}$ generates $CU_{I,A}$ as a topological group. Therefore it suffices to show that $e^{\ri c}\in CU_I$ for every $c\in [I,A]\cap A_{\rm sa}$. But in this case it holds by Lemma \ref{lem: liecom} that $c\in \overline{[A,A]}\cap I_{\rm sa}=\overline{[I,I]}\cap A_{\rm sa}$, which by \cite[Lemma 3.1]{Rob19} again implies that $e^{\ri c}\in CU_I$.  
\end{proof}

The next lemma is essentially \cite[Lemma 3.1]{Rob19} and the same proof works for it. 
\begin{lemma}\label{lem: expi[I,I]}
Let $I$ be a closed two-sided ideal of a unital $C^*$-algebra $A$. 
Then $CU_I$ is generated by $\{e^{{\rm i}x}\mid x\in [I,I]\cap A_{{\rm sa}}\}$ as a topological group.  
\end{lemma}

The next lemma improves Robert's lemma above. 
\begin{lemma}\label{Rob3.1'}
Let $A$ be a unital $C^*$-algebra, $I$ be a closed two-sided ideal of $A$. Then the set $E_{I,A}=\{e^{[a,b]}\mid a\in I_{\rm sa}, b\in A_{\rm sa}\}$ generates $CU_{I,A}$ as a topological group. In particular, $CU_I$ is generated by $\{e^{[a,b]}\mid a,b\in I_{\rm sa}\}$ as a topological group. 
\end{lemma}
\begin{proof}
Let $H$ be the closed subgroup of $U_A$ generated by $E_{I,A}$. Then $H$ is a normal subgroup of $U_A$, since $ue^{[a,b]}u^*=e^{[uau^*,ubu^*]}$ with $uau^*\in I_{\rm sa}$, $ubu^*\in A_{\rm sa}$ if $a\in I_{\rm sa},\,b\in A_{\rm sa}$ and $u\in U_A$.  
By \cite[Lemma 3.1]{Rob19}, the set $E_{I,A}'=\{e^{\ri x}\mid x\in [I,A]\cap A_{\rm sa}\}$ generates $CU_{I,A}$ as a topological group. Because $E_{I,A}\subseteq E'_{I,A}$, this implies that $H\subseteq CU_{I,A}$. Thus, we only have to show that $CU_{I,A}\subseteq H$. 
Let $u\in U_I$ and $v\in U_A$, and we show that $(u,v)\in H$. 
For this purpose, we recall the formula (essentially the same as \cite[(2.1)]{Rob19} that for elements $x,y,z$ in a group $G$, 
\begin{equation}
    (z,xy)=(z,x)x(z,y)x^{-1}.\label{(2.1)'Robert}
\end{equation}
Since for every $r>0$ the exponentials of skew-adjoint operators of norm at most $r>0$ in $I$ (resp. $A$) algebraically generates $U_I$ (resp. $U_A$), by the formulas \cite[(2.1)]{Rob19} and (\ref{(2.1)'Robert}), we may assume that $u,v$ are of the form $u=e^{\ri a},\,v=e^{\ri b}$ for $a\in I_{\rm sa}$ and $b\in A_{\rm sa}$ 
 such that their norms are so small so that the Campbell--Baker--Hausdorff formula converges.  

This implies that  
\eqa{
(u,v)&=e^{\ri a}e^{-\ri vav^*}\\
&=\exp\left (\ri (a-vav^*)-\frac{\ri}{2}\ri [a,vav^*]+\frac{1}{12}\left ([a,\ri [a,vav^*]+[vav^*,\ri [a,vav^*]]\right )\cdots\right ),
}
Moreover, note that $f(t)=e^{\ri tb}ae^{-\ri tb}$ is a norm-differentiable function in $t\in \R$ with 
$f'(t)=e^{\ri tb}\,\ri [b,a]e^{-\ri tb}$ for every $t\in \R$. Thus, we can perform the vector-valued Riemann integration and see that 
\eqa{
\ri (a-vav^*)&=\ri (a-f(1))=\ri a-\ri \left (f(0)+\int_0^1f'(t)dt\right )\\
&=-\int_0^1 [\underbrace{e^{\ri tb}ae^{-\ri tb}}_{\in I_{\rm sa}},b]dt.
}
This in particular implies, by the standard approximation of vector-valued Riemann integrals by vector-valued step functions, that there exists sequences $x_n, y_n$ all of which are finite sums of elements from the set $\{[c,d]\mid c\in I_{\rm sa},\,d\in A_{\rm sa}\}$, such that $\ri (a-vav^*)=\lim_{n\to \infty}x_n$, and that (by the continuity of $\exp$)
$(u,v)=\lim_{n\to \infty}\exp(x_n+y_n)$. 
By a repeated application of the Trotter product formula, we see that $\exp(x_n+y_n)\in H$ for every $n \in\N$, whence $(u,v)\in H$ as well. This shows the claim. 

\end{proof}

\begin{proposition}\label{prop: idealcommutator}
Let $I$ be a closed two-sided ideal of a unital $C^*$-algebra $A$. Then $\mathcal{Z}(CU_I)=\mathcal{Z}(CU_A)\cap CU_I$ holds. 
\end{proposition}
\begin{proof}It is clear that $CU_I\cap \mathcal{Z}(CU_A)\subseteq \mathcal{Z}(CU_I)$, so we prove the converse inclusion. Let $g\in \mathcal{Z}(CU_I)$ and we show that $g$ commutes with every $h\in CU_A$. By Lemma \ref{Rob3.1'}, the set $\{e^{[a,b]}\mid a,b\in A_{\rm{sa}}\}$ generates $CU_A$ as a topological group, we may assume that $h=e^{[a,b]}$ where $a,b\in A_{\rm{sa}}$. Let $(e_{\lambda})_{\lambda\in \Lambda}$ be an approximate unit in $I$ that is quasi-central in $A$. 

In order to prove $gh=hg$, it is sufficient to show that $(g-1)x=x(g-1)$, where $x=[a,b]$. Because $g-1\in I$ and $(e_{\lambda})$ is an approximate unit for $I$ that is quasicentral in $A$, we have (use also $g\in \mathcal{Z}(CU_I)$)
\begin{align*}
(g-1)(ab-ba)&=\lim_{\lambda}(g-1)e_{\lambda}^4(ab-ba)
=\lim_{\lambda}(g-1)(e_{\lambda}ae_{\lambda}^2be_{\lambda}-e_{\lambda}be_{\lambda}^2ae_{\lambda})\\
&=\lim_{\lambda}\frac{d}{dt}\bigg|_{t=0}(g-1)\underbrace{\exp(t[e_{\lambda}ae_{\lambda},e_{\lambda}be_{\lambda}])}_{\in CU_I}\\
&=\lim_{\lambda}\frac{d}{dt}\bigg|_{t=0}\exp(t[e_{\lambda}ae_{\lambda},e_{\lambda}be_{\lambda}])(g-1)\\
&=\lim_{\lambda}[e_{\lambda}ae_{\lambda},e_{\lambda}be_{\lambda}](g-1)\\
&=\lim_{\lambda}(ab-ba)e_{\lambda}^4(g-1)\\
&=(ab-ba)(g-1).
\end{align*}
Thus $gx=xg$, whence $g\in \mathcal{Z}(CU_A)$ holds.
\end{proof}

\begin{lemma}\label{lem: commutator and ideal}
Let $A$ be a unital $C^*$-algebra $H,K$ be normal subgroups of $U_A$. Then ${\rm Id}_A([H,K])={\rm Id}_A((H,K)-1)$ holds. 
\end{lemma}
\begin{proof}
Let $h\in H,k\in K$. Then 
$$hk-kh=(hkh^{-1}k^{-1}-1)kh,$$
which shows that $[H,K]\subseteq {\rm Id}_A((H,K)-1)$. Therefore, ${\rm Id}_A([H,K])\subseteq {\rm Id}_A((H,K)-1)$ holds. From the identity 
$$hkh^{-1}k^{-1}-1=(hk-kh)h^{-1}k^{-1}$$
we see that $(H,K)-1\subseteq {\rm Id}_A([H,K])$, whence ${\rm Id}_A((H,K)-1)\subseteq {\rm Id}_A([H,K])$.
\end{proof}
\begin{lemma}\label{lem:idealCU_I}
Let $A$ be a unital $C^*$-algebra and $I$ be a closed two-sided ideal of $A$. Then $\mathrm{Id}_A([I,I])=\mathrm{Id}_A([CU_I,CU_I])$.
\end{lemma}
\begin{proof}

By Lemma \ref{lem: commutator and ideal} and since $CU_I$ is topologically perfect, we have 
\eqa{
    {\rm Id}_A([CU_I,CU_I])&={\rm Id}_A((CU_I,CU_I)-1)\\
    &={\rm Id}_A(\overline{(CU_I,CU_I)}-1)\\
    &={\rm Id}_A(CU_I-1)={\rm Id}_A((U_I,U_I)-1)\\
    &={\rm Id}_A([U_I,U_I])={\rm Id}_A([I,I]).
}
\end{proof}

\subsection{Perfect Ideals}
\begin{definition}
Let $I$ be a closed two-sided ideal of a unital $C^*$-algebra $A$. We call $I$ {\it perfect} if $I={\rm Id}_A([I,I])$ holds. 
\end{definition}
\begin{remark}
Let $A$ be a $C^*$-algebra, and $I$ be a closed two-sided ideal of $A$. Then ${\rm Id}_A([I,I])={\rm Id}_I([I,I])$ holds. Indeed, it is clear that ${\rm Id}_I([I,I])\subseteq {\rm Id}_A([I,I])$. On the other hand, for all $a,b\in A$ and $x,y\in I$, we have 
\[a[x,y]b=\lim_{j\to \infty}ae_j[x,y]e_jb\in {\rm Id}_{I}([I,I]),\]
where $(e_j)_{j\in J}$ is an approximate unit for $I$. Since ${\rm Id}_A([I,I])$ is the closed linear span of the elements of this type, we obtain ${\rm Id}_A([I,I])={\rm Id}_I([I,I])$. From this observation, it is safe to write ${\rm Id}([I,I])$ instead of ${\rm Id}_A([I,I])$. We will follow this convention in the sequel. 
\end{remark}
\begin{proposition}\label{prop: I' is perfect} Let $I$ be a closed two-sided ideal of a unital $C^*$-algebra $A$. Then $I':={\rm Id}([I,I])$ is perfect. Moreover, it is the biggest perfect ideal contained in $I$.
\end{proposition}

\begin{lemma}\label{lem: criterion for perfectness}
Let $A$ be a (not necessarily unital) $C^*$-algebra. Then the following two conditions are equivalent. 
\begin{list}{}{}
\item[(1)] $A$ is perfect in itself: $A={\rm Id}([A,A])$. 
\item[(2)] There is no nonzero $*$-homomorphism (i.e., a character) $A\to \C$.
\end{list}
\end{lemma}
\begin{proof}
(1)$\implies$(2) Any $*$-homomorphism $\chi\colon A\to \C$ must vanish on $[A,A]$, hence on ${\rm Id}([A,A])=A$.\\
(2)$\implies$(1) If (1) does not hold, then $I={\rm Id}([A,A])$ is a proper closed two-sided ideal of $A$, and $B=A/I$ is then a nonzero abelian $C^*$-algebra. Thus, any pure state $\chi\colon B\to \C$ is a nonzero $*$-homomorphism. By composing $\chi$ with the quotient map $A\to B$, we obtain a nonzero $*$-homomorphism $A\to \C$, whence (2) fails. 
\end{proof}
\begin{proof}[Proof of Proposition \ref{prop: I' is perfect}]
It is clear that the only $*$-homomorphism ${\rm Id}([I,I])\to \C$ is the trivial homomorphism, so $I'$ is perfect by Lemma \ref{lem: criterion for perfectness}.

The `moreover' part is clear.
\end{proof}

\begin{remark}
Even though the set $\{e^{\ri x}\mid x\in [A,A]_{\rm sa}\}$ generates $CU_A$ as a topological group, it is not in general true that the Lie algebra of $CU_A$ is $\overline{[A,A]_{\rm sa}}$. For example, let $A$ be the hyperfinite II$_1$ factor, and fix $\theta\in [0,1]\setminus \Q$ and consider unitaries $u,v\in U(A)=U_A$ satisfying $(u,v)=e^{2\pi \ri\theta}1$. 
Then because $\theta$ is irrational, $CU_A$ contains $\T$. 
If $\overline{[A,A]}_{\rm sa}$ were the Lie algebra of $CU_A$, then for each $s\in \R$, the condition $e^{\ri st}\in CU_A\,(t\in \R)$ implies that $s=\ri^{-1}\frac{d}{dt}\big|_{t=0}e^{\ri st}\in \overline{[A,A]}_{\rm sa}$, whence $\overline{[A,A]}_{\rm sa}\supset \R$, which is a contradiction (the unique tracial state on $A$ vanishes on $\overline{[A,A]}_{\rm sa}$, but not on $\R$. Note that by \cite[Proposition 2.7]{CP79}, $A_{\rm sa}=\overline{[A,A]}_{\rm sa}+\R1$ holds). 
\end{remark}

\section{Normal subgroup structure of $U_A$}\label{sect4}
Here the core of the paper starts. In this section we prove the results concerning the normal subgroup structure of $U_A$, hereby proving half of Theorem A.
\begin{proposition}\label{prop:perfectnormalsubgroups}
Let $H$ be a closed normal subgroup of $U_A$, for a unital $C^*$-algebra $A$. Then $\overline{(H,H)}$ is a topologically perfect closed normal subgroup of $U_A$ - the biggest topologically perfect subgroup contained in $H$ - and it is of the form $CU_I$, where $I\subseteq A$ is a closed two-sided ideal.
\end{proposition}
\begin{proof}
Let $H$ be a closed normal subgroup of $U_A$. We prove that $\overline{(H,H)}$ is topologically perfect. It will be automatically normal, since it is normal in $H$ and $H$ is normal in $U_A$, so by \cite[Theorem 3.7]{Rob19} it is normal in $U_A$. It will be also the biggest subgroup with such a property, since every closed topologically perfect subgroup $H'\leq H$ must be contained in $\overline{(H,H)}$, which follows from the equality $H'=\overline{(H',H')}$.

Set $I:={\rm{Id}}([H,A])$. We show \[\overline{(H,H)}=\overline{(H,U_A)}=CU_I,\] and this will finish the proof since $CU_I$ is topologically perfect. In fact, it is enough to show that $\overline{(H,U_A)}=CU_I$ since the former will be then a topologically perfect subgroup of $H$ which thus must be contained in $\overline{(H,H)}$, and obviously contains $\overline{(H,H)}$. However, by \cite[Theorem 3.3]{Rob19} we have $\overline{(H,U_A)}=\overline{(U_I,U_A)}$ and the latter is equal to $CU_I$ by Lemma~\ref{lem:U_IU_I=U_AU_A}.
\end{proof}

\begin{lemma}\label{lem:CU_I=CU_I'}
Let $I$ be a closed two-sided ideal of a unital $C^*$-algebra $A$. Then $CU_I=CU_{I'}$, where $I'={\rm{Id}}([I,I])$.
\end{lemma}
\begin{proof}
By Lemma \ref{lem: expi[I,I]}, the set $\{e^{\ri x}\mid x\in [I,I]\cap A_{\rm sa}\}$ generates $CU_I$ as a topological group and if $x\in [I,I]\cap A_{\rm sa}\subseteq I'$ then clearly $e^{\ri x}\in U_{I'}$. Therefore, we have $CU_I\subseteq U_{I'}$. By Proposition~\ref{prop:perfectnormalsubgroups}, $CU_{I'}=\overline{(U_{I'},U_{I'})}$ is the biggest topologically perfect normal subgroup of $U_{I'}$. On the other hand, $CU_I$ is a topologically perfect normal subgroup of $U_{I'}$, thus $CU_I\subseteq CU_{I'}$. Since $I'\subseteq I$, we have also $CU_{I'}\subseteq CU_I$, and so $CU_I=CU_{I'}$ holds.
\end{proof}
We will need the next well-known result. We include the proof for completeness. 
\begin{lemma}\label{lem: U_0(A) generates A}
Let $A$ be a unital $C^*$-algebra. Then $U_A$ spans $A$. 
\end{lemma}
\begin{proof}
Since $A_{\rm sa}$ spans $A$, it suffices to show that any self-adjoint contraction $x\in A$ is in the linear span of $U_A$, and note that the standard argument produces the unitary $u(x)=x+\ri (1-x^2)^{\frac{1}{2}}$ such that $x=\frac{1}{2}(u(x)+u(x)^*)$, and $u(x)$ is homotopic to 1 by the continuous path $[0,1]\ni t\mapsto e^{-\frac{\pi \ri}{2}(1-t)}u(tx)\in U(A)$. 
\end{proof}

\begin{lemma}\label{lem:CU_Iinjective}
Let $I$ and $J$ be perfect ideals of a unital $C^*$-algebra $A$. If $CU_I=CU_J$, then $I=J$.
\end{lemma}
\begin{proof}
Let $I$ and $J$ be perfect ideals such that $CU_I=CU_J$. We show that $I=J$. 
Assume by contradiction that $I\not \subseteq J$. 
By the perfectness of $I$ and $I\not\subseteq J$, we have $[I,I]\not \subseteq J$. 
Note that since $U_I$ spans $\widetilde{I}$ by Lemma \ref{lem: U_0(A) generates A}, 
the set $\{xu-ux\mid x\in I_{\rm sa}, u\in U_I\}$ spans $[I,I]$.
Thus there exists $x\in I_{{\rm sa}}$ and $u\in U_I$ such that $ux-xu\in [I,I]\setminus J$, whence $uxu^*-x=u(xu^*)-(xu^*)u\in [I,I]\setminus J$. 
We may also assume that $\|x\|$ is small enough so that $uxu^*-x$ lies in a neighborhood of $0$ in $A$ where $\exp$ is a local homeomorphism. Then by Lemma \ref{lem: expi[I,I]}, we have  $v:=\exp(\ri (uxu^*-x))\in CU_I$, however $v\notin U_J$, thus $v\notin CU_J$, which is a contradiction. Therefore, $I\subseteq J$ holds, and by symmetry, the same argument shows $J\subseteq I$ as well. 
\end{proof}

\begin{theorem}\label{thm:perfectsubgrpsofU_A}
There is a one-to-one bijection between perfect ideals $I\subseteq A$ and topologically perfect closed normal subgroups of $U_A$ of the form $I\to CU_I$.
\end{theorem}
\begin{proof}
For every two-sided ideal $I\subseteq A$, $CU_I$ is a topologically perfect closed normal subgroup of $U_A$.

By Lemma~\ref{lem:CU_Iinjective}, the map $I\to CU_I$ is injective when restricted to perfect ideals. 
We check that the map $I\to CU_I$ is onto the set of all topologically perfect closed normal subgroups of $U_A$. Let $H$ be a topologically perfect closed normal subgroup of $U_A$. Then $H=\overline{(H,H)}$, where the right hand side is by Proposition~\ref{prop:perfectnormalsubgroups}  equal to $CU_I$ for some ideal $I\subseteq A$. But then for $I':=\mathrm{Id}([I,I])$ we have that $I'$ is perfect by Proposition~\ref{prop: I' is perfect} and $CU_I=CU_{I'}$ by Lemma~\ref{lem:CU_I=CU_I'}.
\end{proof}
\begin{definition}
For a closed two-sided ideal $I$ of a unital $C^*$-algebra $A$, we define $N_I=\{u\in U_A\mid (u,U_A)\subseteq \widetilde{I}\}$. 
\end{definition}
\begin{lemma}\label{lem:N_Iequivdef}
Let $A$ be a unital $C^*$-algebra, $I$ be a closed two-sided ideal of $A$. Then $N_I=\{u\in U_A\mid (u,U_A)\subseteq CU_I\}$ holds.
\end{lemma}
\begin{proof}
Denote the set on the right hand side by $M_I$. Clearly, $M_I\subseteq N_I$. The converse follows immediately from \cite[Lemma 3.2]{Rob19} and Lemma \ref{lem:U_IU_I=U_AU_A}.
\end{proof}
As an immediate consequence of Lemmas~\ref{lem:CU_I=CU_I'} and~\ref{lem:N_Iequivdef} we get the following corollary.
\begin{corollary}\label{cor: N_I=N_I'}
Let $A$ be a unital $C^*$-algebra, $I$ be a closed two-sided ideal of $A$. Then $N_I=N_{I'}$ holds.
\end{corollary}

\begin{theorem}\label{thm:setofnormalsubgrps}
Let $A$ be a unital $C^*$-algebra. The set of all closed normal subgroups of $U_A$ is equal to the disjoint union \[\bigsqcup_{I\subseteq A{{\rm \,\,perfect\, ideal}}} \{H\le U_A\,\mid\, CU_I\leq H\leq N_I\}.\]

\end{theorem}
\begin{proof}
Let $H\subseteq U_A$ be a closed normal subgroup. By \cite[Theorem 3.3]{Rob19} the ideal $I={\rm Id}([H,A])$ satisfies $CU_{I,A}=\overline{(H,U_A)}\subseteq H$. 
By Lemmas~\ref{lem:U_IU_I=U_AU_A} and ~\ref{lem:CU_I=CU_I'}, we have $CU_{I,A}=CU_I=CU_{I'}$, and $I'$ is a perfect ideal by Proposition \ref{prop: I' is perfect}. Moreover, the equality $\overline{(H,U_A)}=CU_I$, together with Lemmas~\ref{lem:N_Iequivdef} and~\ref{cor: N_I=N_I'} implies that $H\subseteq N_I=N_{I'}$.  

Conversely, suppose that $H$ is a closed subgroup of $U_A$ such that for some perfect ideal $I$ of $A$, we have $CU_I\leq H\leq N_I$. We claim that $H$ is normal in $U_A$. Pick $h\in H$ and $u\in U_A$. Since $H\subseteq N_I$ we have $uhu^*h^*\in CU_I$, thus $uhu^*\in CU_Ih\subseteq H$.

It remains to show that the union from the statement is taken of disjoint sets. Suppose that $H\subseteq U_A$ is a closed normal subgroup and there are two perfect ideals $I$ and $J$ such that $CU_I\leq H\leq N_I$ and $CU_J\leq H\leq N_J$. Then since $(N_I,N_I)\subseteq CU_I=\overline{(CU_I,CU_I)}$ by the definition of $N_I$ and since $CU_I$ is topologically perfect, we have  $\overline{(N_I,N_I)}=CU_I=\overline{(CU_I,CU_I)}$ and $\overline{(N_J,N_J)}=CU_J=\overline{(CU_J,CU_J)}$. Therefore, $\overline{(H,H)}$ must be equal to both $CU_I$ and $CU_J$, showing that $CU_I=CU_J$, which implies, by Lemma~\ref{lem:CU_Iinjective}, that $I=J$.
\end{proof}

\begin{remark}
As pointed out by the referee, the idea of placing every closed normal subgroup $H\leq U_A$ in between the closed normal subgroups $CU_I\leq H\leq N_I$ resembles the notion of \emph{embracing} a Lie ideal $L\subseteq A$ by a closed two-sided ideal $I\subseteq A$ from \cite{BKS08} (see \cite[Definition 1.5]{BKS08}. One could say that Theorem~\ref{thm:setofnormalsubgrps} proves that every closed normal subgroup of $U_A$ is embraced by a perfect ideal of $A$. An analogous result for $V_A$ is proved in Theorem~\ref{thm:normalsubgrpsofV_A}.
\end{remark}

Robert proves in \cite[Corollary 3.8]{Rob19} that when a unital $C^*$-algebra is simple, then $CU_A/\Ze(CU_A)$ is topologically simple. The converse is not true though. Theorem~\ref{thm:setofnormalsubgrps} however allows us to completely characterize when $CU_A/\Ze(CU_A)$ 
is topologically simple. We provide this characterization in Theorem~\ref{thm:simplicity} in the next section jointly with the characterization when $V_A$ is topologically simple.

\section{Normal subgroup structure of $V_A$}\label{sect5}
This section is parallel to the previous one. Here we prove the results concerning the normal subgroup structure of $V_A$, having then proved the second half of Theorem A. We also provide a characterization of those unital $C^*$-algebras satisfying that $CU_A/\Ze(CU_A)$ and $V_A$ are topologically simple.

We start with the following lemma that appears in \cite[Lemma 2.1]{AD23}.
\begin{lemma}\label{lem:approxInnquotient}
Let $A$ be a unital $C^*$-algebra. For every $\phi\in V_A$ and every closed two-sided ideal $I$ we have $\phi[I]\subseteq I$, thus $\phi$ naturally induces $\pi_I\circ \phi=:P_I(\phi)\in V_{A/I}$, where $\pi_I:A\rightarrow A/I$ is the quotient map. The map $P_I: V_A\rightarrow V_{A/I}$ is a continuous homomorphism satisfying $P_I(\mathrm{Ad}(u))=\mathrm{Ad}(\pi_{I}(u))$.
\end{lemma}

Notice that $V_I\subseteq \ker(P_I)$. The converse is in general not clear.

The following theorem, concerning the normal subgroup structure of $V_A$, is analogous to Theorem~\ref{thm:setofnormalsubgrps} that concerns the normal subgroup structure of $U_A$.

\begin{theorem}\label{thm:normalsubgrpsofV_A}
Let $A$ be a unital $C^*$-algebra. The set of all closed normal subgroups of $V_A$ is equal to the disjoint union \[\bigsqcup_{I\subseteq A{{\rm \,\,perfect\, ideal}}} \{H\le V_A\mid V_I\leq H\leq \ker (P_I)\}.\]
\end{theorem}
We need a few lemmas for the proof. The first one is an immediate corollary to \cite[Corollary~3.8 and Lemma~3.9]{Rob19} (it is mentioned in \cite[Theorem 3.10]{Rob19}). 
\begin{lemma}[\cite{Rob19}]\label{lem V_I is perfect}
Let $A$ be a unital $C^*$-algebra $I$ be a closed two-sided ideal. Then $V_I$ is topologically perfect and the equality $V_I=\overline{{\rm Ad}(CU_I)}$ holds. 
\end{lemma}
\begin{proof} By \cite[Lemma~3.9]{Rob19} and Lemma \ref{lem:U_IU_I=U_AU_A}, $V_I=\overline{{\rm Ad}(CU_{I,A})}=\overline{{\rm Ad}(CU_I)}$, and by \cite[Theorem~3.6]{Rob19}, $CU_I$ is topologically perfect. Thus 
\[\overline{(V_I,V_I)}=\overline{({\rm Ad}(CU_I),{\rm Ad}(CU_I))}=\overline{{\rm Ad}((CU_I,CU_I))}=\overline{{\rm Ad}(CU_I)}=V_I.\]
\end{proof}

\begin{lemma}\label{lem: centerlessV_A}
Let $A$ be a unital $C^*$-algebra. Then $V_A$ has trivial center. 
\end{lemma}
\begin{proof}
Suppose that $\phi\in \mathcal{Z}(V_A)$. Then for each $u\in \mathcal{U}_A$ and $x\in A$, $${\mathrm{Ad}}(u^*)\circ \phi\circ {\mathrm{Ad}}(u)(x)=\phi(x).$$ The left hand side of the above equality is $u^*\phi(u)\phi(x)\phi(u^*)u=\gamma(u) \phi(x)\gamma(u)^*$, where $\gamma(u)=u^*\phi(u)$. Since $\phi$ is onto and $x$ is arbitrary, it shows that $\gamma(u)$ is in the center of $A$. By \cite[Lemma 3.11]{Rob19}, $\phi$ is trivial.
\end{proof}

The next lemma is analogous to Lemma~\ref{lem:CU_Iinjective}.
\begin{lemma}\label{lem:V_Iinjective}
Let $I$ and $J$ be perfect ideals of a unital $C^*$-algebra $A$. If $V_I=V_J$, then $I=J$.
\end{lemma}
\begin{proof}
Set \[\widetilde{I}:=\mathrm{Id}(\{\phi(x)-x\mid \phi\in V_I=V_J, x\in A\})=\mathrm{Id}(\{uxu^*-x\mid \mathrm{Ad}(u)\in V_I=V_J,x\in A\})=\mathrm{Id}([CU_I,A]),\] where the first equality between the ideals is straightforward and the second follows from the fact that $\mathrm{Ad}[CU_I]$ is dense in $V_I=V_J$ by Lemma \ref{lem V_I is perfect}. 
 Analogously, we set \[\widetilde{J}:=\mathrm{Id}(\{\phi(x)-x\mid \phi\in V_I=V_J, x\in A\})=\mathrm{Id}(\{uxu^*-x\mid \mathrm{Ad}(u)\in V_I=V_J,x\in A\})=\mathrm{Id}([CU_J,A]).\] Notice that we have $\widetilde{I}=\widetilde{J}$, so it suffices to check that $I=\widetilde{I}$ and $J=\widetilde{J}$. We show the former equality. By Lemma~\ref{lem:idealCU_I}, we have \[\widetilde{I}=\mathrm{Id}([CU_I,A])\supset \mathrm{Id}([CU_I,CU_I])=\mathrm{Id}([I,I])=I'=I,\]
since $I$ is perfect. On the other hand, for every $u\in U_I$ and $x\in A$, we have $uxu^*-x\in I$, whence $\widetilde{I}\subseteq I$. Therefore $I=\widetilde{I}$ and $I=J$ holds.
\end{proof}

\begin{proof}[Proof of Theorem~\ref{thm:normalsubgrpsofV_A}]
Let $H$ be a closed normal subgroup of $V_A$. Set \[I:=\mathrm{Id}(\{\phi(x)-x\mid \phi\in H, x\in A\}).\] 
By \cite[Theorem 3.10]{Rob19}, we have $\overline{(H,V_A)}= V_I$. Without loss of generality, we may assume that $I$ is perfect. Indeed, otherwise we replace $I$ by $I':=\mathrm{Id}([I,I])$ which is perfect by Proposition~\ref{prop: I' is perfect}. Then by Lemma~\ref{lem:CU_I=CU_I'} we have $CU_I=CU_{I'}$ and so by Lemma~\ref{lem V_I is perfect}, we get \[V_{I'}=\overline{\{\mathrm{Ad}(u)\mid u\in CU_{I'}\}}=\overline{\{\mathrm{Ad}(u)\mid u\in CU_I\}}=V_I.\] Since $H$ is normal and closed we have $\overline{(H,V_A)}\subseteq H$, thus we get $V_I\leq H\leq NV_I$, where \[NV_I:=\{\phi\in V_A\mid (\phi,V_A)\subseteq V_I\}.\] 

We now claim that $NV_I=\ker (P_I)$. First, we show that $\ker (P_I)\subseteq NV_I$. By repeating the argument from the paragraph above with the closed normal subgroup $H:=\ker (P_I)$ we get that $\ker(P_I)\subseteq NV_J$, where \[J:=\mathrm{Id}(\{\phi(x)-x\mid \phi\in \ker(P_I), x\in A\}).\] However, $J\subseteq I$, so $\ker(P_I)\subseteq NV_J\subseteq NV_I$. Indeed, for every $\phi\in\ker(P_I)$ and $x\in A$ we have (where $\pi_I:A\to A/I$ is the quotient map) \[\pi_I\big(\phi(x)-x\big)=P_I(\phi)\big(\pi_I(x)\big)-\pi_I(x)=0,\] implying that $\phi(x)-x\in I$, which in turn implies that $J\subseteq I$.

Next, we prove the other inclusion $NV_I\subseteq \ker(P_I)$. Pick $\phi\in NV_I$ and $u'\in U_{A/I}$, and let $u\in U_A$ be such that $u'=\pi_I(u)$. Since $(\phi,\mathrm{Ad}(u))\in V_I$ we have that $P_I(\phi)$ and $\mathrm{Ad}(u')$ commute. As $u'\in U_{A/I}$ was arbitrary and $\{\mathrm{Ad}(u)\mid u\in U_{A/I}\}$ is dense in $V_{A/I}$, we get $P_I(\phi)\in Z(V_{A/I})$. By Lemma~\ref{lem: centerlessV_A}, it follows that $P_I(\phi)$ is trivial, in other words, $\phi\in\ker(P_I)$.
Therefore, $\ker (P_I)=NV_I$ holds. We thus obtain $V_I\le H\le \ker (P_I)$.
\medskip

Conversely, if $I$ is a perfect ideal and $H$ is a closed subgroup of $V_A$ such that $V_I\leq H\leq \ker (P_I)$, then we claim that $H$ is normal in $V_A$. Indeed, pick any $\phi\in H$ and $\psi\in V_A$. Then since $\phi\in NV_I$, we have $\psi\phi\psi^{-1}\phi^{-1}\in V_I$, thus $\psi\phi\psi^{-1}\in V_I\phi\subseteq H$, showing that $H$ is normal in $V_A$.

By Lemma~\ref{lem:V_Iinjective}, we have that if $I$ and $J$ are perfect ideals, and $V_I=V_J$, then $I=J$.\medskip

Finally, we show that the union from the statement consists of disjoint sets. This is done similarly as in the proof of Theorem~\ref{thm:setofnormalsubgrps}. Suppose that $H\subseteq V_A$ is a closed normal subgroup and there are perfect ideals $I$ and $J$ such that $V_I\leq H\leq \ker(P_I)$ and $V_J\leq H\leq \ker (P_J)$. Since $\ker (P_I)=NV_I$ and $\ker (P_J)=NV_J$, the quotients $\ker (P_I)/V_I$ and $\ker (P_J)/V_J$ are abelian and we have that $\overline{(H,H)}$ must be equal to both $V_I$ and $V_J$. Therefore, by the argument from the previous paragraph, $I=J$ holds.
\end{proof}

\begin{corollary}\label{cor:perfectsubgrpsofV_A}
There is a one-to-one bijection between perfect ideals $I\subseteq A$ and perfect closed normal subgroups of $V_A$ of the form $I\to V_I$.
\end{corollary}
\begin{proof}
The proof that the map $I\to V_I$ is injective, when restricted to perfect ideals, is given in the proof of Theorem \ref{thm:normalsubgrpsofV_A}. For each perfect ideal $I$ of $A$, $V_I$ is topologically perfect by Lemma \ref{lem V_I is perfect}. Now if $H\leq V_A$ is any topologically perfect closed normal subgroup, then by Theorem~\ref{thm:normalsubgrpsofV_A} there is a perfect ideal $I\subseteq A$ such that $V_I\subseteq H\subseteq \ker (P_I)=NV_I$. Since $NV_I/V_I$ is abelian, it follows that $(NV_I,NV_I)/V_I$ is trivial, so $V_I=\overline{(\ker (P_I),\ker (P_I))}=\overline{(V_I,V_I)}$, thus because $H$ is topologically perfect, we have $H=\overline{(H,H)}=V_I$.
\end{proof}

In view of Theorem \ref{thm:normalsubgrpsofV_A}, it is important to know under which conditions the two groups $V_I$ and $\ker (P_I)=NV_I$ coincide, in which case we have a more precise description of all closed normal subgroups of $V_A$.

A step in this direction is a part of the main result from \cite{AD23} (see \cite[Theorem 2.4]{AD23}) that in particular answers in positive this question for unital locally AF algebras, and the following is therefore an immediate corollary of Theorem~\ref{thm:normalsubgrpsofV_A} and \cite{AD23}. We note that in \cite{AD23}, $U_I$, for an ideal $I$ of a unital $C^*$-algebra $A$, is defined to be the connected component of the unit of unitary group of the unitization $\tilde I$ rather than $\{u\in U_A\mid u-1\in I\}$ as in this paper. The reader can readily verify that these two definitions differ only by the presence of $\T$ and in particular induce the same group $V_I$.
\begin{theorem}\label{thm:AFalgebras}
Let $A$ be a unital locally AF algebra. Then there is a one-to-one correspondence between perfect ideals of $A$ and closed normal subgroups of $V_A$, which is of the form \[I\to V_I.\] In particular, every closed normal subgroup of $V_A$ is topologically perfect.
\end{theorem}

Observe that Theorem~\ref{thm:AFalgebras} proves half of Theorem B from Introduction. We now focus on the other half.

Recall that a unital separable AF algebra $A$ which is a direct limit of a sequence $(A_n)_{n\in\Nat}$ of finite-dimensional algebras admits a convenient description using a Bratteli diagram $\mathcal{D}$. To fix a notation, vertices of $\mathcal{D}$ are denoted by pairs $(p,i)$ which correspond to $i$-th direct summand of $A_p$. We denote by $d(p,i)$ the dimension of this summand. Recall that there may be multiple arrows between vertices $(p,i)$ and $(p+1,j)$ and the number of such arrows corresponds to the multiplicity of the partial embedding between the $i$-th direct summand of $A_p$ and $j$-th direct summand of $A_{p+1}$.  Recall, e.g., from \cite[III.4]{Davidson-book}, that there is a one-to-one correspondence between ideals of $A$ and directed and hereditary subsets of $\mathcal{D}$, and that if $\mathcal{S}\subseteq \mathcal{D}$ is a directed and hereditary subset corresponding to an ideal $I\subseteq A$, then $\mathcal{D}\setminus\mathcal{S}$ is a Bratteli diagram of the quotient $A/I$.

\begin{theorem}\label{thm:Bratteli}
Let $A$ be a separable unital AF algebra with Bratteli diagram $\mathcal{D}$. There is a one-to-one correspondence between closed normal subgroups of $V_A$ and directed hereditary subsets of $\mathcal{D}$ that do not contain a maximal infinite path $(p,n_0)\rightarrow (p+1,n_1)\rightarrow (p+2,n_2)\rightarrow\ldots$ such that for all $q\geq p$ we have $d(q,n_{q-p})=1$.
\end{theorem}
\begin{proof}
We view $A$ as a direct limit of a sequence $(A_n)_{n\in\Nat}$ of finite-dimensional algebras. We retain the notation from the previous paragraph. Applying Theorem~\ref{thm:AFalgebras}, it is equivalent to prove a one-to-one correspondence between perfect ideals of $A$ and directed hereditary subsets of $\mathcal{D}$ that do not contain a maximal infinite path $(p,n_0)\rightarrow (p+1,n_1)\rightarrow (p+2,n_2)\rightarrow\ldots$ such that for all $q\geq p$ we have $d(q,n_{q-p})=1$. Clearly, the multiplicity of each such partial embedding is necessarily 1.

We claim that a directed and a hereditary subset $\mathcal{S}\subseteq \mathcal{D}$ corresponding to an ideal $I\subseteq A$ defines a perfect ideal if and only if $\mathcal{S}$ does not contain a maximal infinite path $(p,n_0)\rightarrow (p+1,n_1)\rightarrow (p+2,n_2)\rightarrow\ldots$ such that for all $q\geq p$ we have $d(q,n_{q-p})=1$. Suppose that $\mathcal{S}$ contains such a maximal path $(p,n_0)\rightarrow (p+1,n_1)\rightarrow (p+2,n_2)\rightarrow\ldots$ and denote by $\mathcal{S}'$ the vertices of this path. We claim that $\mathcal{S}\setminus\mathcal{S}'$ is directed and hereditary. To see this, notice that for every $(q,n_{q-p})\in\mathcal{S}'$ there is exactly one arrow (inside $\mathcal{D}$) from $(q,n_{q-p})$ (necessarily to $(q+1,n_{q+1-p})$) and at most one arrow to $(q,n_{q-p})$. Therefore, if $(r,n_{r-p})\in\mathcal{S}\setminus\mathcal{S}'$ and there is an arrow $(r,n_{r-p})\rightarrow (r+1,n_{r+1-p})$ within $\mathcal{S}$, then $(r+1,n_{r+1-p})\notin \mathcal{S}'$ as otherwise $d(r+1,n_{r+1-p})=1$, thus also $d(r,n_{r-p})=1$ and so $(r,n_{r-p})\in\mathcal{S}'$ because of maximality $\mathcal{S}'$, a contradiction. This verifies directedness. Now if $(r,n_{r-p})\in\mathcal{S}$ is such that for every arrow $(r,n_{r-p})\rightarrow (r+1,n_{r+1-p})$ we have $(r+1,n_{r+1-p})\notin\mathcal{S}'$, then necessarily also $(r,n_{r-p})\notin\mathcal{S}'$ since otherwise, if $(r,n_{r-p})\in\mathcal{S}'$, there is only one arrow from $(r,n_{r-p})$ and by definition it is into a vertex from $\mathcal{S}'$. This verifies the hereditary property. Therefore, $\mathcal{S}\setminus\mathcal{S}'$ defines a subideal $J\subseteq I$ and $I/J$ corresponds to $\mathcal{S'}$ which clearly defines a one-dimensional $C^*$-algebra, i.e. $\Com$. Therefore, $I$ admits a non-trivial character, so it is not perfect by Lemma~\ref{lem: criterion for perfectness}.\medskip

Conversely, suppose that $I$ is not perfect. Then by Lemma~\ref{lem: criterion for perfectness}, $I$ admits a non-trivial character, so there exists a subideal $J\subseteq I$ such that $I/J=\Com$. Let $\mathcal{F}\subseteq \mathcal{S}$ be a directed and hereditary subset that defines $J$. It follows that the quotient $\Com=I/J$ corresponds to the Bratteli subdiagram $\mathcal{S}\setminus\mathcal{F}$. However, $\mathcal{S}\setminus\mathcal{F}$ may correspond to $\Com$ only if for every $(p,n)\in\mathcal{S}\setminus\mathcal{F}$ we have $d(p,n)=1$. Since $\mathcal{S}\setminus\mathcal{F}$ is infinite, by Zorn's lemma we can extract a maximal infinite path $\mathcal{S}'\subseteq\mathcal{S}\setminus\mathcal{F}$. In particular, $\mathcal{S}$ contains a maximal infinite path $(p,n_0)\rightarrow (p+1,n_1)\rightarrow (p+2,n_2)\rightarrow\ldots$ such that for all $q\geq p$ we have $d(q,n_{q-p})=1$. This finishes the proof of the claim.
\end{proof}

\begin{remark}
We remark that such a precise description of closed normal subgroups as it is done for $V_A$ in Theorem~\ref{thm:AFalgebras} is rather impossible for $U_A$ or $CU_A$. Indeed, notice that by Theorem~\ref{thm:setofnormalsubgrps}, for every proper non-trivial ideal $I$ of a unital $C^*$-algebra $A$ we get the closed normal subgroups $CU_I\leq U_I\leq N_I$ of $U_A$ and by Lemma~\ref{lem:imageofNI}, $U_I\subsetneq N_I$ if $\Ze(A/I)$ is non-trivial. In particular, picking $B$ any unital simple AF algebra and setting $A:=B\oplus B\oplus B$ and $I:=B\oplus \{0\}\oplus\{0\}$, we get $\Ze(A/I)$ is two-dimensional, so non-trivial. Therefore, there is no one-to-one correspondence between (perfect) ideals and closed normal subgroups of $U_A$ even when $A$ is AF.
\end{remark}

We conclude this section with the following result that generalizes the main results from \cite{Rob19} that if $A$ is a simple $C^*$-algebra, then $CU_A/\Ze(CU_A)$ and $V_A$ are topologically simple, and employs Theorems~\ref{thm:setofnormalsubgrps} and~\ref{thm:normalsubgrpsofV_A}.
\begin{theorem}\label{thm:simplicity}
Let $A$ be a unital $C^*$-algebra. Then the following are equivalent:

\begin{enumerate}
    \item For every closed two-sided ideal $I$ of $A$, either $I\subseteq \Ze(A)$ or $A/I$ is abelian.
    \item The only perfect ideals of $A$ are the trivial ideal $\{0\}$ and $\mathrm{Id}([A,A])$.
    \item $V_A$ is topologically simple.
    \item $CU_A/\Ze(CU_A)$ is topologically simple.
    \item For every closed two-sided ideal $I$ of $A$, $V_I$ is either trivial or equal to $V_A$.
    \item For every closed two-sided ideal $I$ of $A$, either $CU_I$ is trivial, or $CU_I=CU_A$.
\end{enumerate}
\end{theorem}
\begin{proof}
$(1)\implies (2)$: let $I\subseteq A$ be a perfect ideal. If $I\subseteq \Ze(A)$, then $[I,I]=\{0\}$ and so since $I=\mathrm{Id}([I,I])$ we get that $I=\{0\}$. If $A/I$ is abelian, then necessarily $[A,A]\subseteq I$. It follows from \[I\subseteq\mathrm{Id}([I,I])\subseteq \mathrm{Id}([A,A])\subseteq I\] that $I=\mathrm{Id}([A,A])$. Therefore, (2) holds.\medskip

$(2)\implies (1)$: let $I\subseteq A$ be a closed two-sided ideal and suppose that 
$I\not\subseteq \Ze(A)$. Since by Lemma~\ref{lem: liecom}, $\overline{[I,I]}=\overline{[A,I]}$ and the latter is nonzero as $I\nsubseteq \Ze(A)$, we get that the perfect ideal $I'=\mathrm{Id}([I,I])$ is nonzero. By the assumption we must have that $I'=\mathrm{Id}([A,A])$ but then $[A,A]\subseteq I$, so $A/I$ is abelian.  Therefore (1) holds.\medskip

$(2)\implies (3)$: since $V_{\{0\}}=\{\mathrm{Id}\}$ and $\ker(P_{\{0\}})=\{\mathrm{Id}$\}, and also $V_{\mathrm{Id}([A,A])}=V_A$ and $\ker(P_{\mathrm{Id}([A,A])})=V_A$, it immediately follows from Theorem~\ref{thm:normalsubgrpsofV_A} that the only closed normal subgroups of $V_A$ are the trivial subgroup and $V_A$ itself, thus $V_A$ is topologically simple.\medskip

$(3)\implies (2)$: Let $I\subseteq A$ be a perfect ideal. Then $V_I\leq V_A$ is a perfect closed normal subgroup and thus $V_I$ is either trivial or equal to $V_A$ which by Corollary~\ref{cor:perfectsubgrpsofV_A} gives that $I$ is either trivial or equal to $\mathrm{Id}([A,A])$.\medskip

$(2)\implies (4)$: let $H'\leq CU_A/\Ze(CU_A)$ be a closed normal subgroup. By the fourth isomorphism theorem, there is a closed normal subgroup $\Ze(CU_A)\leq H\leq CU_A$ such that $H'=H/\Ze(CU_A)$. By \cite[Theorem 3.7]{Rob19}, $H$ is normal in $U_A$, so by Theorem~\ref{thm:setofnormalsubgrps}, there is a perfect ideal $I\subseteq A$ such that $CU_I\leq H\leq N_I$. If $I=\{0\}$ then $N_I=\Ze(U_A)$ and so $H\subseteq CU_A\cap \Ze(U_A)\subseteq \Ze(CU_A)$, thus $H=\Ze(CU_A)$. If $I=\mathrm{Id}([A,A])$, then by Lemma~\ref{lem:CU_I=CU_I'}, $CU_I=CU_A$ and since $CU_I\leq H\leq CU_A$ we get that $H=CU_A$. Therefore, $CU_A/\mathcal{Z}(CU_A)$ is topologically simple.\medskip

$(4)\implies (2)$: we prove the contraposition. 
Suppose there is a perfect ideal $I\subseteq A$ that is neither equal to $\{0\}$ nor equal to $\mathrm{Id}([A,A])$. Then by Lemma~\ref{lem:CU_Iinjective}, $CU_I$ is neither equal to $CU_{\{0\}}=\{1\}$ nor equal to $CU_{{\rm Id}[A,A]}=CU_A$, and thus $CU_I$ is a non-trivial proper closed normal subgroup of $CU_A$. 
Let $H_I$ be the closed normal subgroup of $CU_A$ generated by $CU_I$ and $\mathcal{Z}(CU_A)$, which is the closure of $CU_I\cdot \Ze(CU_A)$. Set $H_I'=H_I/\Ze(CU_A)$.

We claim that $H'_I$ is a proper non-trivial closed normal subgroup of $CU_A/\Ze(CU_A)$. $H'_I$ must be non-trivial as otherwise, we have $CU_I\subset \mathcal{Z}(CU_A)$. Therefore, $CU_I=\overline{(CU_I,CU_I)}\subset \overline{(\mathcal{Z}(CU_A),\mathcal{Z}(CU_A))}=\{1\},$ a contradiction.
 
It remains to show that $H'_I\neq CU_A/\Ze(CU_A)$. 
Since $\overline{(H_I,H_I)}=\overline{(CU_I,CU_I)}=CU_I\neq CU_A=\overline{(CU_A,CU_A)}$, we see that $H_I\subsetneq CU_A$, and we obtain $H_I'\neq CU_A/\mathcal{Z}(CU_A)$.  
Therefore, $(4)$ does not hold.
\medskip

$(5)\implies (3)$: let $H$ be a closed normal subgroup of $V_A$. Then 
by Theorem~\ref{thm:normalsubgrpsofV_A}, there exists a perfect ideal $I\subseteq A$ such that $V_I\le H\le \ker(P_I)$. If $V_I=V_A$ then clearly $H=V_A$, and if $V_I=\{\mathrm{Id}\}$, then $I=\{0\}$, so $\ker(P_I)=\{\mathrm{Id}\}$, and thus $H=\{\mathrm{Id}\}$. Thus, $V_A$ is topologically simple.\medskip

$(3)\implies (5)$: let $I\subseteq A$ be any closed two-sided ideal. Then $V_I$ is a closed normal subgroup of $V_A$. Since $V_A$ is topologically simple, $V_I$ must be either trivial or equal to $V_A$, proving (5).\medskip

$(4)\implies (6)$: let $I\subseteq A$ be a closed two-sided ideal of $A$. 
Let $H_I$ be the closed subgroup of $CU_A$ generated by $CU_I$ and $\mathcal{Z}(CU_A)$, which is the closure of $CU_I\cdot\Ze(CU_A)$. Then $H_I'=H_I/\mathcal{Z}(CU_A)$ is a closed normal subgroup of the topologically simple group $CU_A/\mathcal{Z}(CU_A)$, which is, as showed in the proof of $(4)\Rightarrow (2)$, topologically isomorphic to $CU_I/\Ze(CU_I)$. If $H_I'=\{1\}$, then $CU_I\subseteq \mathcal{Z}(CU_I)$, so $CU_I$ is abelian and topologically perfect, whence $CU_I=\{1\}$. If $H_I'=CU_A/\mathcal{Z}(CU_A)$, then  $H_I=CU_A$, so $$CU_A=\overline{(H_I,H_I)}=\overline{(CU_I\cdot\Ze(CU_A),CU_I\cdot\Ze(CU_A))}=\overline{(CU_I,CU_I)}=CU_I,$$ so $(6)$ holds.\medskip

$(6)\implies (4)$: let $H'\leq CU_A/\Ze(CU_A)$ be a closed normal subgroup. By the third isomorphism theorem, there exists a closed normal subgroup $\Ze(CU_A)\le H\le CU_A$ such that $H'=H/\Ze(CU_A)$. By \cite[Theorem 3.7]{Rob19} $H$ is normal in $U_A$ so by Theorem~\ref{thm:setofnormalsubgrps}, there exists a perfect ideal $I\subseteq A$ such that $CU_I\leq H\leq N_I$. Now we use the assumption. If $CU_I=CU_A$, then $H=CU_A$ and $H'=CU_A/\Ze(CU_A)$, so assume that $CU_I$ is trivial $=CU_{\{0\}}$. However, then $I=\{0\}$ by the perfectness of $I$ and Corollary \ref{cor:perfectsubgrpsofV_A}, and arguing as in the implication $(2)\Rightarrow(4)$ we get that $H\subseteq \Ze(CU_A)$, thus $H=\Ze(CU_A)$ and so $H'$ is trivial.
Therefore, $CU_A/\mathcal{Z}(CU_A)$ is topologically simple.
\end{proof}

\section{Lie structure of $U_A$ and applications}\label{sect6}

In this section, we shall further investigate the Lie structure of $U_A$. The main application, proved in Corollary~\ref{cor:AdUA-inv-algebras}, is that $C^*$-subalgebras that are invariant under unitaries from the connected component of the identity are invariant under all unitaries.

The following lemma is already implicit in \cite{Rob19} and it immediately follows from \cite[Lemma 3.2]{Rob19} and Lemma~\ref{lem:U_IU_I=U_AU_A}.

\begin{lemma}\label{lem:imageofNI}
Let $A$ be a unital $C^*$-algebra and $I\subseteq A$ be a two-sided closed ideal. Let $P_I:A\to A/I$ be the canonical quotient map. Then $N_I=P_I^{-1}(\Ze(U_{A/I}))\cap U_A$.
\end{lemma}

Let us define $\mathfrak{n}_I$ to be the set \[\{a\in A_{{\rm sa}}\mid [a, A]\subseteq \overline{[I,I]}\}.\]

\begin{lemma}\label{lem:lieofni}
We have \[\mathfrak{n}_I=\{a\in A_{{\rm sa}}\mid \exp(\ri ta)\in N_I\;\forall t\in\Rea\};\] i.e. $\mathfrak{n}_I$ is the Lie algebra of $N_I$.
\end{lemma}
\begin{proof}
Denote the set on the right hand side from the statement as $X$.

First, let us show $\mathfrak{n}_I\subseteq X$. Pick $a\in\mathfrak{n}_I$ and $u\in U_A$. We may suppose that $\|a\|$ is small enough so that $\log\big((u,e^{\ri a})\big)=\log \big(\exp(\ri uau^*)\exp(-\ri a)\big)$ is defined and equal to $\ri c$ for some $c\in A_{{\rm sa}}$, which is moreover a converging series given by the Baker-Campbell-Hausdorff formula.

By definition, we have $ua-au\in \overline{[I,I]}\subseteq I$, and thus also $uau^*-a\in I$. Since also, by the Baker-Campbell-Hausdorff formula, and since by definition $[a,A]\in \overline{[I,I]}$, we have \[\log \big(\exp(\ri uau^*)\exp(-\ri a)\big)-\ri (uau^*-a)\in I,\] we obtain $\log \big(\exp(\ri uau^*)\exp(-\ri a)\big)\in I$, so $c\in I$. It follows that $(\exp(\ri a),u)\in \tilde{I}$, so $\exp(\ri a)\in N_I$. Clearly, $\exp(\ri ta)\in N_I$, for any $t\in \Rea$ with $|t|\leq 1$. Since $N_I$
 is a group, $\exp(\ri ta)\in N_I$ for any $t\in \R$, and $\mathfrak{n}_I\subseteq X$ holds.\medskip

Now we show that $X\subseteq \mathfrak{n}_I$. Pick some $a\in X$ and $u\in A$. We check that $[a,u]\in \overline{[I,I]}$. Since the linear span of $U_A$ is $A$ by Lemma~\ref{lem: U_0(A) generates A}, we may suppose that $u\in U_A$. We may also suppose that $\|a\|$ is small enough so that $\log\big(\exp(\ri (uau^*-a))\big)$ is defined and equal to $\ri (uau^*-a)$.

Since $a\in X$, by applying the Trotter product formula, we obtain
\[\exp(\ri (uau^*-a))=\lim_{n\to\infty} \big((\exp(\ri a/n),u)\big)^n\in CU_I.\]

It follows that $uau^*-a\in I$, and thus also $ua-au=[u,a]\in I$. Since $[A,A]\cap I\subseteq\overline{[I,I]}$ by Lemma~\ref{lem: liecom}, we are done.
\end{proof}

\begin{proposition}\label{prop:lieofni}
$N_I$ is a Banach-Lie group with $\mathfrak{n}_I$ as the Lie algebra.
\end{proposition}
\begin{proof}
We show that for every $a\in A_{{\rm sa}}$ such that $\|a\|$ is sufficiently small we have $a\in\mathfrak{n}_I$ if and only if $\exp(\ri a)\in N_I$ since then we will have that $\exp$ induces a local homeomorphism between $\mathfrak{n}_I$ and $N_I$ and we can conclude by \cite[Definition 5.32 and Corollary 5.34]{HofMor}.

By Lemma~\ref{lem:lieofni}, if $a\in\mathfrak{n}_I$ then $\exp(\ri a)\in N_I$. So conversely assume that $e^{\ri a}\in N_I$ and that $\|a\|$ is small enough so that $\log(e^{\ri a})$ and $\log(e^{P(\ri a)})$ are defined (and equal to $\ri a$ and $P(\ri a)$ respectively). We show that $a\in\mathfrak{n}_I$. Again by Lemma~\ref{lem:lieofni}, it suffices to show that $e^{\ri ta}\in N_I$, for all $t\in\Rea$.
It is clear that in fact it is enough to prove that for $t$ such that $|t|\leq 1$. Let $P_I: A\rightarrow A/I$ be the quotient map. Then by Lemma~\ref{lem:imageofNI} we have \[P_I(e^{\ri a})\in \Ze(U_{A/I})= \Ze(A/I)\cap U_{A/I},\] 

where the last equality holds since by Lemma~\ref{lem: U_0(A) generates A}, $U_{A/I}$ spans $A/I$.

Consequently, \[P_I(a)=P_I\big(\log(e^{\ri a})/\ri\big)=\log\big(P_I(e^{\ri a})\big)/\ri=\in \Ze(A/I)\]

and thus also $P_I(ta)\in \Ze(A/I)$, for any $t$ such that $|t|\leq 1$. 

Therefore $\exp(\ri P_I(ta))\in \Ze(A/I)\cap U_{A/I}=\Ze(U_{A/I})$ and since $P_I(\exp(\ri ta))=\exp(\ri P_I(ta))$ we get that $\exp(\ri ta)\in N_I$ by Lemma~\ref{lem:imageofNI}. 

\end{proof}

\begin{lemma}\label{lem:AdU(A)normality}
Every closed connected Lie normal subgroup of $U_A$ is normal in $U(A)$.
\end{lemma}
\begin{proof}
Let $H$ be a closed connected Lie normal subgroup of $U_A$. By Theorem~\ref{thm:setofnormalsubgrps} there is a perfect ideal $I$ of $A$ such that $CU_I\leq H\leq N_I$. Recall from the proof of Theorem~\ref{thm:setofnormalsubgrps} that we have $[H,A]\subseteq I$. One can check directly from their respective definitions that both $CU_I$ and $N_I$ are normal in $U(A)$. It suffices to verify that for every $u\in U(A)$ and $g\in H$ we have $(u,g)\in CU_I$ since then $ugu^*\in CU_I\cdot H\subseteq H$.

Let us first check $(u,g)\in\tilde{I}$. This follows since $ug-gu\in I$ as $[H,A]\subseteq I$, thus for $x=ug-gu\in I$ we have $ugu^*g^*=(x+gu)u^*g^*=xu^*g^*+1\in\tilde{I}$.

Now since $H$ is connected and Lie, $g=e^{\ri b_1}\cdots e^{\ri b_n}$, for some self-adjoint elements $b_1,\ldots,b_n$ from the Lie algebra of $H$. Without loss of generality we can suppose that $g=e^{\ri b}$, for $\|b\|$ small enough, so that $c:=\log (u,g)$ is defined and the Baker--Campbell--Hausdorff formula for $e^{\ri ubu^*}e^{-\ri b}$ converges. 
Thus, $(u,e^{\ri b})=e^{\ri c}.$ 
The rest of the argument is the same as \cite[Lemma 3.2]{Rob19}. We include the argument for completeness. For each $t\in \R$, $e^{\ri tb}\in H$, so we differentiate $(u,e^{\ri tb})\in \tilde{I}$ at $t=0$ to obtain $ubu^*-b\in \tilde{I}$, which implies $ubu^*-b\in I$ (as in \cite[Lemma 3.2]{Rob19}, we use the fact (see e.g., \cite[Theorem 13.6]{RudinFunctionalAnalysisMR1157815}) that a commutator in a unital Banach algebra cannot be a nonzero scalar multiple of the unit).   
Then 
$$\log (u,e^{\ri b})-\ri(ubu^*-b)=\log (e^{\ri ubu^*}e^{-\ri b})-\ri\underbrace{(ubu^*-b)}_{\in I}\in \overline{[A,H]}\subset I,$$
whence $\ri c\in I\cap \overline{[A,A]}=\overline{[I,I]}$ and thus $e^{\ri c}\in CU_I$ by Lemma \ref{lem: expi[I,I]}. 
\end{proof}
\begin{corollary}\label{cor:AdUA-inv-algebras}
Let $A$ be a unital $C^*$-algebra and let $B\subseteq A$ be a $C^*$-subalgebra that is $\mathrm{Ad}_{U_A}$-invariant. Then $B$ is $\mathrm{Ad}_{U(A)}$-invariant.
\end{corollary}
\begin{proof}
Let $A$ and $B$ be as in the statement. Then $U_B$ is a closed normal subgroup of $U_A$, which is moreover connected and Lie. Thus by Lemma~\ref{lem:AdU(A)normality}, $U_B$ is normal in $U(A)$, i.e. it is $\mathrm{Ad}_{U(A)}$-invariant. Since $U_B$ spans $B$ by Lemma~\ref{lem: U_0(A) generates A},  also $B$ is $\mathrm{Ad}_{U(A)}$-invariant.
\end{proof}

\section{Group characterizations of simplicity of $C^*$-algebras}\label{sect7}
In this final section, we investigate conditions on the structure of $U_A$ that guarantee simplicity of $A$. That is, in a way, a reverse problem to what we consider in Theorem~\ref{thm:simplicity}. We find applications to reduced group $C^*$-algebras.
\subsection{General conditions on simplicity of $C^*$-algebras}
\begin{lemma}(well-known)\label{lem: center is trivial}
Let $A$ be a unital simple $C^*$-algebra. Then $\mathcal{Z}(A)=\C$.  
\end{lemma}
\begin{proof} Let $a\in \mathcal{Z}(A)\setminus \{0\}$. Consider $I=\overline{Aa}=\overline{aA}$, which is a closed two-sided ideal of $A$. Since $A$ is unital, $a\in I$, so $I\neq \{0\}$. Therefore, by the simplicity of $A$, $I=A$ holds. Thus, there exists $b\in A$ such that $\|ba-1\|<1$. Then $ba$, hence $a$ is invertible. This shows that any nonzero element in $\mathcal{Z}(A)$ is invertible, whence $\mathcal{Z}(A)=\C$ by Banach--Mazur theorem. 
\end{proof}

The following result of Cuntz and Pedersen from \cite[Proposition 2.7]{CP79} will be used often, so we state it here.
For a unital $C^*$-algebra $A$, we denote by $T(A)$ the set of all tracial states on $A$. 
\begin{theorem}[\cite{CP79}]\label{thm:CuntzPedersen}
Let $A$ be a $C^*$-algebra. If $A$ does not admit a tracial state, then $\overline{[A,A]}=A$ holds. If $A$ admits a tracial state, then $\overline{[A,A]}=\bigcap_{\tau\in T(A)} \ker(\tau)$ holds.  
\end{theorem}
\begin{lemma}\label{lem UA=CUA condition}
Let $A$ be a unital $C^*$-algebra. If $A$ does not admit a tracial state, then $U_A=CU_A$, and if $A$ admits exactly one tracial state, then $U_A/CU_A$ is trivial or equal to $\T$.
\end{lemma}
\begin{proof}
If $A$ does not admit a tracial state,
then by Theorem~\ref{thm:CuntzPedersen}, $\overline{[A,A]}=A$, thus by the combination of Lemma~\ref{lem: liecom} and \cite[Lemma 3.1]{Rob19} we get that $CU_A=U_A$. Suppose that $A$ admits exactly one tracial state. Then again by Theorem~\ref{thm:CuntzPedersen} we get $A=\overline{[A,A]}\oplus\Com$. Denote by $Q$ the quotient map $U_A\to U_A/CU_A$. Let $u\in U_A$. Write $u$ as $\exp(\ri a_1)\cdots\exp(\ri a_n)$, where $a_1,\ldots,a_n\in A_{{\rm sa}}$. Write $a_k=b_k+r_k1$, $b_k\in \overline{[A,A]}_{\rm sa}, r_k\in \R$ for each $k=1,\dots,n$. 
Then \[u=e^{\ri b_1}\cdots e^{\ri b_n}\exp(\ri r)\]
where $r=\sum_{k=1}^{n}r_k$, so $Q(u)=Q(e^{\ri r})$ because $e^{\ri b_k}\in CU_{A}$ by \cite[Lemma 3.1]{Rob19}. Consequently, $Q[U_A]=Q[\T]$. Since $Q$ is continuous and $U_A$ is connected, $U_A/CU_A$ is a connected quotient of $\T$, thus either trivial, or isomorphic to $\T$.
\end{proof}

\begin{theorem}
Let $A$ be a unital $C^*$-algebra with at most one tracial state which we then moreover require not to be a character. Then $A$ is simple if and only if $CU_A/\Ze(CU_A)$ is topologically simple and $\Ze(CU_A)\subseteq \T$.
\end{theorem}

\begin{proof}
Suppose first that $A$ is simple. By \cite[Corollary 3.8]{Rob19}, $CU_A/\Ze(CU_A)$ is topologically simple. Set $H:=\Ze(CU_A)$. It is a closed normal subgroup of $CU_A$ and by \cite[Theorem 3.7]{Rob19}, it is also normal in $U_A$. Therefore, by Theorem~\ref{thm:setofnormalsubgrps}, there exists a perfect ideal $I\subseteq A$ such that $CU_I\leq H\leq N_I$. If $I=\mathrm{Id}([A,A])$, then $H=CU_A$, implying that $CU_A$ is abelian and topologically perfect, whence trivial.  

So suppose that $I=\{0\}$. Then $H\leq N_I=\Ze(U_A)$. Since $A$ is simple, by Lemma~\ref{lem: center is trivial}, $\Ze(A)=\Com$, so $\Ze(U_A)=\T$ (cf. Lemma~\ref{lem: U_0(A) generates A}) and we have $\Ze(CU_A)\subseteq \T$.\medskip

Now we prove the converse. Let $I\subseteq A$ be a non-zero closed two-sided ideal. By Theorem~\ref{thm:simplicity}, either $I\subseteq \Ze(A)$ or $A/I$ is abelian.\medskip

\noindent{\bf Case 1.} $I\subseteq \Ze(A)$. 
Assume first that $A$ admits a unique tracial state $\tau$.
Pick a non-zero $b\in I_{{\rm sa}}$. By Theorem~\ref{thm:CuntzPedersen}, we can write $b=a+r1$, where $a\in\overline{[A,A]}$ and $r\in\Rea$. Since $b\in \Ze(A)$, also $a\in \Ze(A)$. Then by \cite[Lemma 3.1]{Rob19}, for every $t\in \R$, 
\[\exp(\ri ta)\in CU_A\cap \Ze(U_A)\subseteq \Ze(CU_A)\subseteq \T.\] 
By differentiating the left hand side at $t=0$, we obtain $a\in \R$, which implies $a\in \R\cap \overline{[A,A]}=\{0\}$. Therefore, for every $b\in I_{{\rm sa}}$ we have $b=r\in I\cap \C^{\times}$. Thus $I=A\subseteq \mathcal{Z}(A)$, so $A$ is an abelian $C^*$-algebra with a unique tracial state. Thus $A=\C$, which admits a character, a contradiction.

Next, assume that $A$ does not admit a tracial state. Then by Theorem~\ref{thm:CuntzPedersen} and Lemma~\ref{lem UA=CUA condition}, we have $A=\overline{[A,A]}$, $U_A=CU_A$, and $\mathcal{Z}(U_A)=\mathcal{Z}(CU_A)\subseteq \T$. Thus $\mathcal{Z}(U_A)=\T$. This implies that $\mathcal{Z}(A)=\C$. Since $A$ does not admit a tracial state, $A\neq \C$ and thus $I$ is a proper ideal of $A$ contained in $\C$, proving that $I=\{0\}$, which is a contradiction.
\medskip

\noindent{\bf Case 2.} $A/I$ is abelian. It follows that $\overline{[A,A]}\subseteq I\subseteq A$. Assume that $A$ admits a unique tracial state $\tau$.

Since again by Theorem~\ref{thm:CuntzPedersen}, $A=\overline{[A,A]}\oplus \Com$ and if $I\neq A$ we get that $I=\overline{[A,A]}$. 
Thus, it follows that $A/I=\Com$ and so $A$ admits a non-trivial character, a contradiction. Thus $I=A$. 
\medskip

Next, assume that $A$ does not have a tracial state. Then by Theorem~\ref{thm:CuntzPedersen}, $A=\overline{[A,A]}=I$.

This concludes that $A$ is simple.
\end{proof}

\subsection{Simplicity of reduced group $C^*$-algebras}
Let $G$ be a countable discrete group, $C_{r}^*(G)$ be the reduced group $C^*$-algebra of $G$. Denote by ${\rm Rep}(G)$ the class of all unitary representations of $G$, and ${\rm Rep}_{\lambda}(G)$ the subclass of all $\pi\in {\rm Rep}(G)$ with $\pi\prec \lambda$. Here, $\lambda$ is the left regular representation of $G$ and $\prec$ denotes the weak containment (we refer to \cite[Part II]{BdHV} for a background on unitary representations of groups, weak containment, and its relations with group $C^*$-algebras). For $\pi\in {\rm Rep}(G)$ acting on a Hilbert space $H$ and for any $x\in \C G$ (the group algebra of $G$), we define 
\[\pi(x)=\sum_{g\in G}a_g\pi(g),\,\,x=\sum_{g\in G}a_g\cdot g\,\,\,a_g\in \C.\]
We will be using the following characterization of the weak containment of representations. 
\begin{proposition}\label{prop weakcontainment}
    For $\pi_1,\pi_2\in {\rm Rep}(G)$, the following conditions are equivalent. 
    \begin{list}{}{}
    \item[(1)] $\pi_1\prec \pi_2$.
    \item[(2)] $\|\pi_1(x)\|\le \|\pi_2(x)\|$ for every $x\in \C G$.
    \item[(3)] There exists a unital $*$-homomorphism $\Phi\colon C_{\pi_2}^*(G)\to C_{\pi_1}(G)$ such that $\Phi(\pi_2(g))=\pi_1(g)$ for every $g\in G$. 
    \end{list}
\end{proposition}
See \cite[Theorem F.4.4]{BdHV} for a proof. The next result is well-known. We include the proof for completeness. 
\begin{proposition}\label{prop:amenablecharacterization}
Let $G$ be a countable discrete group. The following conditions are equivalent. 
\begin{list}{}{}
\item[(1)] 
There exists a unital surjective $*$-homomorphism from $C_r^*(G)$ onto a unital abelian $C^*$-algebra.
\item[(2)] There exists a unital surjective $*$-homomorphism from $C_r^*(G)$ onto $\C$. 
\item[(3)] $G$ is amenable. 
\end{list}
\end{proposition}
\begin{proof}
(3)$\implies$(2) Assume that $G$ is amenable and let $\iota_G$ be the trivial representation of $G$. Then $\iota_G\prec \lambda_G$, whence there exists a unital surjective $*$-homomorphism from $C_r^*(G)$ onto $\C=C^*_{\iota_G}(G)$.
\\

(2)$\implies$(3) Assume that there exists a unital surjective $*$-homomorphism $\Phi\colon C_r^*(G)\to \C$. Then $\pi(g)=\Phi(\lambda(g)),\,\,(g\in G)$ defines a unitary representation $\colon G\to \T$. Let $N=\ker (\pi)$. 
Then the restriction $\lambda_G|_N$ of the left regular representation $\lambda_G$ of $G$ to the subgroup $N$ is unitarily equivalent to a multiple of $\lambda_N$ and $\pi|_N$ is unitarily equivalent to a multiple of $\iota_N$. Thus for every $x\in \C N$, we have 
\[\|\iota_N(x)\|=\|\pi(x)\|\le \|\lambda_G(x)\|=\|\lambda_N(x)\|,\]
which shows that $\iota_N\prec \lambda_N$, whence $N$ is amenable. 
Since subgroups of $\T$ are abelian hence amenable, and amenability passes to extensions, it follows that $G$ is also amenable.\\

(2)$\implies$(1) is trivial.\\
(1)$\implies$(2) Unital abelian $C^*$-algebras have pure states, and they are nonzero $*$-homomorphisms onto $\C$. 
\end{proof}

Recall that an infinite group is called \emph{ICC} (infinite conjugacy classes) if all the non-trivial conjugacy classes are infinite. Recall also that every locally compact group contains a largest normal amenable subgroup called the \emph{amenable radical}. It is known for a discrete group $G$ that if $C^*_r(G)$ is simple, then the amenable radical of $G$ is trivial (see \cite[Proposition 3]{dlH} and that if $G$ has trivial amenable radical (which by Breuillard--Kalantar--Kennedy--Ozawa theorem from \cite{BKKO} is equivalent to $C_r^*(G)$ having the unique tracial state), then $G$ is ICC (see e.g., \cite[Theorem 4.3]{Hag17}). Finally, the ICC property of $G$ is equivalent to $C_r^*(G)$ being prime (see e.g., \cite[Proposition 2.3]{MR1989500Murphy2003}). 
We summarize the above results.
\begin{theorem}[\cite{BKKO,Hag17,dlH,MR1989500Murphy2003}]\label{thm Cstar simple implication}
Let $G$ be a countable discrete group. Consider the following conditions. 
\begin{list}{}{}
\item[(1)] $C_r^*(G)$ is simple.
\item[(2)] $C_r^*(G)$ has unique tracial state.
\item[(3)] $G$ has trivial amenable radical.
\item[(4)] $G$ is {\rm ICC}.
\item[(5)] $C_r^*(G)$ is prime. 
\end{list}
Then (1)$\implies$(2)$\iff$(3)$\implies$(4)$\iff$(5).
\end{theorem}

\begin{proposition}\label{prop:ICCgroupalgebras}
Let $G$ be a countable discrete ICC group. Then for $A:=C^*_r(G)$ we have that all proper quotients of $A$ are abelian if and only if $CU_A/\Ze(CU_A)$ is topologically simple. Moreover, in that case $\Ze(CU_A)\subseteq \Ze(U_A)=\T$.
\end{proposition}

We will need the following well-known lemma. 
\begin{lemma}\label{lem:ICC}
     Let $G$ be a countable discrete ICC group, $A=C^*_r(G)$ be the reduced group $C^*$-algebra of $G$. Then $\mathcal{Z}(A)=\C$.      
\end{lemma}

\begin{proof} By Theorem \ref{thm Cstar simple implication}, $C_r^*(G)$ is prime, and any unital prime C$^*$-algebra has trivial center (see e.g., \cite[Lemma 1.2.47]{MR1940428AraMathieu}
). Alternatively, a proof using the 2-norm expansion of $z\in \mathcal{Z}(C_r^*(G))$ inside the group von Neumann algebra $L(G)$ is possible as well. \end{proof}

\begin{proof}[Proof of Proposition \ref{prop:ICCgroupalgebras}]
First, notice that by Lemma~\ref{lem:ICC}, $\Ze(A)=\Com$ which is equivalent with $\Ze(U_A)=\T$.

Suppose first that all proper quotients of $A$ are abelian. Let $H'\le CU_A/\Ze(CU_A)$ be a closed normal subgroup and, by the fourth isomorphism theorem, let $\Ze(CU_A)\leq H\le CU_A$ be a closed normal subgroup such that $H'=H/\Ze(CU_A)$. By \cite[Theorem 3.7]{Rob19}, $H$ is also normal in $U_A$. Therefore, by Theorem~\ref{thm:setofnormalsubgrps} there exists a perfect ideal $I\subseteq A$ such that $CU_I\leq H\leq N_I$. 

Suppose that $I\neq\{0\}$. Because $A/I$ is abelian, $\overline{[A,A]}\subseteq I$ holds. It follows, by \cite[Lemma 3.1]{Rob19}, that $CU_A\leq U_I$ and thus, since by Proposition~\ref{prop:perfectnormalsubgroups}, $CU_I$ is the largest closed perfect normal subgroup contained in $U_I$ , that $CU_A=CU_I$. So since $CU_I\leq H\leq CU_A$ we get $H=CU_A$. 

Suppose that $I=\{0\}$. Then $\Ze(CU_A)\leq H\leq N_I=\Ze(U_A)=\T$. Since $\Ze(U_A)\cap CU_A\subseteq \Ze(CU_A)$ we get $H= \Ze(CU_A)$ and so $H'$ is trivial.

Therefore, $CU_A/\Ze(CU_A)$ is topologically simple.
\medskip

Now we conversely assume that $CU_A/\Ze(CU_A)$ is topologically simple. By Theorem~\ref{thm:simplicity} we get that for every nonzero proper ideal $I\subseteq A$ we have that either $I\subseteq \Ze(A)$ or $A/I$ is abelian. However, the former is not possible since $\Ze(A)=\Com$ by Lemma~\ref{lem:ICC}.\medskip

Notice that by the first part of the proof, with $H'$ being the trivial subgroup, we get that $\Ze(CU_A)\subseteq\Ze(U_A)=\T$. 

\end{proof}

\begin{lemma}\label{lem:simplegrpalgebra}
Let $G$ be a countable discrete group and suppose that $C_r^*(G)$ is simple. Then  $U_A/\T=CU_A/\Ze(CU_A)$.
\end{lemma}
\begin{proof}
 By Theorem \ref{thm Cstar simple implication}, $G$ is necessarily ICC and so by Proposition~\ref{prop:ICCgroupalgebras}, $\Ze(CU_A)\subseteq\Ze(U_A)=\T$. Moreover, by Theorem \ref{thm Cstar simple implication} again, $C_r^*(G)$ has a unique tracial state and so by Theorem~\ref{thm:CuntzPedersen}, $A=\overline{[A,A]}\oplus\Com$.
 
 We define a topological isomorphism $T:U_A/\T\to CU_A/\Ze(CU_A)$. For $u\in U_A$, we denote by $[u]_1$ its equivalence class in $U_A/\T$, and for $u\in CU_A$, we denote by $[u]_2$ its equivalence class in $CU_A/\Ze(CU_A)$. Let $[u]_1\in U_A/\T$ be an element with a representative $u\in U_A$. Then $u=\exp(\ri (a_1+r_1))\cdots\exp(\ri(a_n+r_n))$, where $a_1,\ldots,a_n\in \overline{[A,A]}$ and $r_1,\ldots,r_n\in\Rea$. We set \[T([u]_1):=[\exp(\ri a_1)\cdots\exp(\ri a_n)]_2.\] 
 
 Then $T$ is well-defined. Indeed, if $u=e^{\ri (a_1+r_1)}\cdots e^{\ri (a_n+r_n)}$ and $u'=e^{\ri (a_1'+r_1')}\cdots e^{\ri (a_m'+r_m')}$ are elements in $U_A$ which defines the same class in $U_A/\T$ where $a_1,\dots,a_n,a_1',\dots,a_m'\in \overline{[A,A]}_{\rm sa}$ and $r_1,\dots,r_n,r_1',\dots,r_m'\in \R$, then there exists $z\in \T$ such that $u'=zu$, whence 
 \[e^{\ri a_1'}\cdots e^{\ri a_m'}=we^{\ri a_1}\cdots e^{\ri a_n},\]
 where $w=\frac{z\lambda}{\lambda'}, \lambda=\exp(\ri (r_1+\dots +r_n))$, $\lambda'=\exp(\ri (r_1'+\dots+r_m')\in \T$. In particular, $w\in CU_A\cap \T= \mathcal{Z}(CU_A)$ holds. Therefore, 
 $$[e^{\ri a_1}\cdots e^{\ri a_n}]_2=[e^{\ri a_1'}\cdots e^{\ri a'_m}]_2,$$
 whence $T$ is well-defined. 

 It is clear that T is a surjective group homomorphism. In order to check that T is a homeomorphism, it is enough to check that T is injective and continuous, since then $T^{-1}$ is a Borel homomorphism between Polish groups, thus continuous (see \cite[Theorem 9.10]{Kechrisbook}).
 
 To show that $T$ is injective, pick $[u]_1\neq 1\in U_A/\T$, where $u\in U_A$. Then $u\notin\T$, so writing $u=\exp(\ri (a_1+r_1))\cdots\exp(\ri(a_n+r_n))$, where $a_1,\ldots,a_n\in \overline{[A,A]}$ and $r_1,\ldots,r_n\in\Rea$, and then defining $u':=\exp(\ri a_1)\cdots\exp(\ri a_n)$, we get that $u'\neq 1$. Moreover, $u'\notin\T$ and since $\Ze(CU_A)\subseteq\T$ we get \[T([u]_1)=[u']_2\neq 1.\] Since $[u]_1\in U_A/\T$ was arbitrary, we get that $T$ is injective. Finally, we check the continuity. Suppose $u_n\in U_A$ satisfies $[u_n]_1\to [1]$ in $U_A/\T$. Without loss of generality, we may assume that $u_n\to 1_{U_A}$. For each $n$, write $u_n$ as 
 $e^{\ri (a_1+r_1)}\cdots e^{\ri (a_m+r_m)}$ and denote by $u'_n\in CU_A$ the corresponding element $e^{\ri a_1}\cdots e^{\ri a_n}$. Notice that for each $n$, we have $u_n=u'_n t_n$, where $t_n\in\T$. By compactness of $\T$ and passing to a subsequence if necessary, we may suppose that $t_n\to t\in \T$. Since $u_n\to 1_{U_A}$ we get $u'_n\to t\in\T$. Then $t\in CU_A\cap\T=\Ze(CU_A)$ and since $T([u_n]_1)=[u'_n]_2$ for each $n$, we get $T([u_n]_1)\to CU_A/\Ze(CU_A)$. This finishes the proof.
\end{proof}

\begin{theorem}\label{thm:simplicity of C_r(G)}
Let $G$ be a countable discrete non-amenable ICC group. Then $A:=C^*_r(G)$ is simple if and only if $U_A/\T$ is topologically simple.
\end{theorem}
\begin{proof}
If $A$ is simple, then by Proposition~\ref{prop:ICCgroupalgebras}, $CU_A/\Ze(CU_A)$ is topologically simple and so by Lemma~\ref{lem:simplegrpalgebra}, $U_A/\T$ is topologically simple.

Conversely, suppose that $U_A/\T$ is topologically simple. We claim that $CU_A/\Ze(CU_A)$ is then topologically simple. Indeed, let $H'$ be a closed normal subgroup of $CU_A/\Ze(CU_A)$ and let $\Ze(CU_A)\leq H\leq CU_A$ be a closed normal subgroup such that $H'=H/\Ze(CU_A)$. Again by \cite[Theorem 3.7]{Rob19}, $H$ is normal in $U_A$, so either $H\subseteq \T$ or $H=U_A$. 
Notice that this argument, applied to $H=\Ze(CU_A)$, in particular implies that $\Ze(CU_A)\subseteq\T$ as the case $\Ze(CU_A)=U_A$ is impossible.\medskip

\noindent{\bf Case 1.} $H=U_A$: this implies that $CU_A=U_A$, and thus $CU_A/\Ze(CU_A)=CU_A/T=U_A/\T$ is topologically simple.\medskip

\noindent{\bf Case 2.} $H\subseteq\T$: It follows that $H$ is central in $CU_A$ and thus $H=\Ze(CU_A)$, so $H'$ is trivial.\medskip

We have therefore proved that $CU_A/\Ze(CU_A)$ is topologically simple.
 So by Proposition~\ref{prop:ICCgroupalgebras}, all the proper quotients of $C_r^*(G)$ are abelian. However, since $G$ is not amenable, by Proposition~\ref{prop:amenablecharacterization}, $C_r^*(G)$ is simple.
\end{proof}
\section*{Acknowledgments}
We would like to thank the anonymous referee for careful reading and suggestions which improved the presentation of the paper. 
\bibliographystyle{siam}
\bibliography{references}
\end{document}